\documentclass[12pt]{article}

\usepackage{authblk}
\usepackage{fleqn,amsfonts,amssymb,graphicx,theorem}
\usepackage{caption}
\usepackage{mathtools}
\usepackage{extarrows}
\usepackage{adjustbox}

\usepackage{ xcolor}

\newcommand{\errata}[1]{}

\newcommand{\dd}{\mathop{}\!{d}}

\newcommand{\restrict[1]}{\raisebox{-.5ex}{$|$}_{#1}}

\newcommand{\Z}{\mathbb Z}
\newcommand{\N}{\mathbb N}
\newcommand{\C}{\mathbb C}

\newcommand{\R}{\mathbb R}
\newcommand{\Q}{\mathbb Q}

\newcommand{\mat}[2][ccccccccc]{\left( \begin{array}{#1} #2\\ \end{array} \right)}

\newcommand{\set}[1]{\left\{#1\right\}}
\newcommand{\HP}{\mathbb H}

\newcommand{\vsni}{\vskip 0.2cm}

\newenvironment{proof}{Proof.:}{\hfill$\square$}
\newtheorem{theorem}{Theorem}
[section]

\newtheorem {lemma}[theorem]{Lemma}
\newtheorem {proposition}[theorem]{Proposition}
\newtheorem {corollary}[theorem]{Corollary}
\theorembodyfont{\normalfont}
\newtheorem {definition}[theorem]{Definition}

\newtheorem {remark}[theorem]{Remark}

\def\be{\begin{equation}}
\def\ee{\end{equation}}

\newcommand{\beqn}{\begin{eqnarray}}
\newcommand{\eeqn}{\end{eqnarray}}
\newcommand{\beqno}{\begin{eqnarray*}}
\newcommand{\eeqno}{\end{eqnarray*}}

\renewcommand {\l}{\left}
\newcommand {\ri}{\right}
\renewcommand  {\Re}{{\rm Re}}

\newcommand {\bG}{{\mathbb G}}

\newcommand{\bem}{\l(\! \begin{array}}
\newcommand{\eem}{\end{array}\!\ri)}
\newcommand{\bsm}{\left(\begin{smallmatrix}} % begin small matrix
\newcommand{\esm}{\end{smallmatrix}\right)}  % end   small matrix

\newcommand{\northeast}
{\raisebox{1.5mm}{$\lrcorner$}\raisebox{-2mm}{$\ulcorner$}}

\begin{document}

\title{Words and numbers: a dynamical systems perspective}

\author{Stefano Isola and Francesco Marchionni}

\affil {School of Science and Technology, University of Camerino,

via Madonna delle Carceri, 62032 Camerino, Italy,

e-mail: stefano.isola@unicam.it, francesco.marchionni@studenti.unicam.it}

%\address[Stefano Isola]{School of Science and Technology\\
%University�of Camerino\\
%via Madonna delle Carceri, 62032 Camerino}
%\email[Stefano Isola]{stefano.isola@unicam.it}

\date{}

\maketitle

\begin{abstract}
\noindent
Along with some known and less known results, we discuss new insights relating combinatorics of words and the ordering of the rationals from a dynamical systems point of view, somehow continuing along the path started in \cite{BI}.
We obtain in particular a set of results that structure and enrich the correspondence between the Stern-Brocot (SB) ordering of rational numbers and the corresponding ordering of Farey-Christoffel (FC) words, a class of words that, since their appearance in literature
at the end of the 18th century, have revealed numerous relationships with other fields of mathematics. 
Among the results obtained here is the construction of substitution rules that act on the FC words in a parallel way to the maps on the positive reals that generate the permuted SB tree both vertically and horizontally. We further show that these rules naturally induce a map of the space of (infinite) Sturmian sequences into itself. Finally, a complete correspondence is obtained between the vertical and horizontal motions on the SB tree and the geodesic motions along scattering geodesics and the horocyclic motion along Ford circles in the upper half-plane, respectively.
\end{abstract}

\newpage
\section{Preliminaries}
\noindent
The Stern-Brocot (SB) tree $\cal T$ is binary rooted tree which provides a way to order (and thus to count) the elements of $\Q_+$, the set of positive rational numbers, so that every number appears (and thus is counted) exactly once (see \cite{St}, \cite{Br}, \cite{GKP}, \cite{BI}). To begin with, we say that a pair of nonnegative fractions 
$\frac a b<\frac {c} {d}$ is a {\sl Farey pair} if the unimodular relation $bc-ad=1$ holds (so that their distance is $1/bd$). The basic operation needed to construct $\cal T$ associates to each Farey pair their
{\sl mediant} $${\frac ab}\oplus {c\over d} = {a+c \over b+d}$$
One readily sees that the {\sl child} ${\frac ab}\oplus {c\over d}$ always lies somewhere in between its {\sl parents}
 $\frac ab$ and $\frac cd$, forming Farey pairs with them.  Moreover, among all the fractions lying strictly between ${a\over b}$ and  ${c\over d}$  it is the one (and only one) with the smallest denominator, and is always in lowest terms whenever the  parents do (see \cite{Ri}). 
 
 \vskip 0.2cm
\begin{remark}
Note that the mediant operation arises naturally in the following way:  let  $L$ be the vertical half-line $\{x=1, y\geq 0\}$ in $\R^2$, and denote by $U$ the subspace of $\R^2$ given by of all vectors $u=(q,p)$ with positive integer coordinates. Let $T:U\to \Q_+$ be the map given by $T(q,p)=p/q$, that is the ordinate of the intersection of $u$ with $L$. Each reduced fraction on $L$ is thus the image with $T$ of a vector of $U$ with coprime coordinates. Finally, given $u_1,u_2 \in U$, we have
$$
T(u_1+u_2)=T(u_1)\oplus T(u_2)
$$
\end{remark}
 \vskip 0.2cm
 \noindent
Now, taking as initial pair  $\frac 0 1$ and $\frac 1 0$, we take their mediant  $\frac 1 1$ as the {\sl root}  of the tree. Then one writes  one generation after the other using the above operation (a portion of this structure is depicted in Figure \ref{fig:The Stern-Brocot tree}). As already observed, $\Q_+$ and ${\cal T}$ are in bijection. To a given $x\in \Q_+$, we associate its {\sl depth}, as the level of ${\cal T}$ it belongs to.
\begin{lemma}\label{uno} {\rm (\cite{BI}, Lemma 1.2)} Let $x\in \Q_+$ then
$$
\qquad x=[a_0;a_1,\dots ,a_n] \quad \Longrightarrow \quad {\rm depth}(x)= \sum_{i=0}^n a_i
$$
\end{lemma}

\begin{figure}[htbp]
\begin{center}
\includegraphics[width=\textwidth, trim = 4cm 17.3cm 4cm 4.4cm , clip]{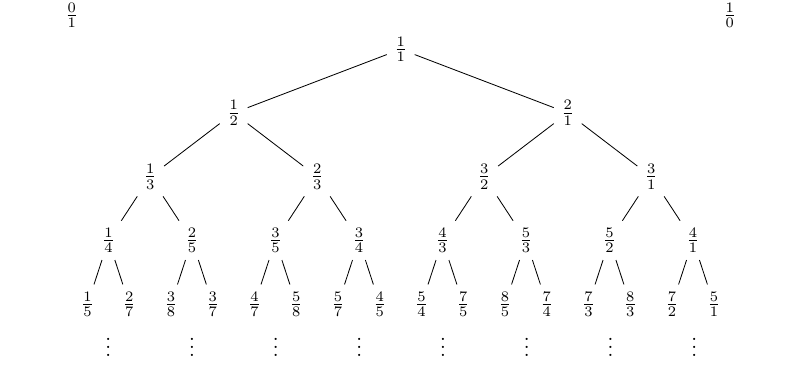}
\caption{The first five levels of the Stern-Brocot tree }
\label{fig:The Stern-Brocot tree}
\end{center}
\end{figure}

\begin{remark}\label{subtree}
Note that the sub-tree $\cal S$ of $\cal T$ having $\frac 1 2$ as root node and vertex set $\Q_+\cap [0,1]$ (sometimes called Farey tree) can be obtained exactly in the same way as $\cal T$ taking as initial pair $\frac 0 1$ and $\frac 1 1$ instead of $\frac 0 1$ and $\frac 1 0$. One easily sees (\cite{BI}, Lemma 1.1) that $\phi ({\cal T}) = \cal S$ where $\phi :  [0,\infty) \to [0,1]$ is the
invertible map defined by $\phi (\infty)=1$ and
$\phi (x) = \frac x {x+1}$.
\end{remark}

\vsni
\noindent
One can also construct an equivalent tree whose  vertex set is formed by binary strings, each fraction $p/q \in \cal T$  corresponding to a binary word $w_{\frac pq}$ obtained by concatenation of its left and right parent as follows\footnote{Defining FC words by reversed concatenation does not really change matters. In particular, it is easy to show by induction that FC words defined as above (resp. by reversed concatenation) are also {\sl Lyndon words}, i.e. they are minimal (resp. maximal) w.r.t. cyclic permutations. We should also notice that what we call here Farey-Christoffel words, to emphasize their relation with the Farey order of the rationals, are commonly called just {\sl Christoffel words}  \cite{BLRS} since they have been studied for the first time by Christoffel in 1875, see \cite{Ch}.
}.
\begin{definition}\label{fw} {\sl(Farey-Christoffel (FC) words)} Set 
$$
w_{\frac 01}=0\quad \hbox{and} \quad w_{\frac 10}=1
$$
If moreover $\frac {p'} {q'}$ and $\frac {p''} {q''}$ is a Farey pair and $\frac {p}{q}={\frac {p'} {q'}}\oplus {p''\over q''}$, we define
$$
w_{\frac {p}{q}} = w_{\frac {p'}{q'}} \, w_{\frac {p''}{q''}} 
$$
\end{definition}
\noindent
{\bf Some notations:}  for $s\in \{0,1\}$ set ${\hat s} =1-s$. Then, for $w\in \{0,1\}^*$ given by $w=s_1\dots s_n$ we set
$${\hat w} ={\hat s}_1 \dots {\hat s}_n\quad \hbox{and} \quad {\tilde w}=s_n\dots s_1$$ Also denote by $|w|$ the length of $w$ and by $|w|_s$ the number of occurrence of the symbol $s\in \{0,1\}$ in $w$. 
\vsni
\noindent
The above construction establishes a one to one correspondence between $\Q_+ \simeq {\cal T}$ and the set $\cal F$ of FC words.
\begin{theorem}\label{propo}
We have the following properties:

\begin{enumerate}
\item given $w \in \cal F$, we have $w=w_{\frac pq}$ with $\frac pq = \frac {|w|_1}{|w|_0}$ (so that $|w|=p+q$) ;
\item given  $\frac p q \in {\cal T}$ with $p+q>1$ we have $w_{\frac pq}=0 \,c \, 1$ for some $c\in \{0,1\}^*$ satisfying $c={\tilde c}\,$;
\item given  $w_{\frac pq}=0 \,c\, 1$, we have $w_{\frac qp} = 0\,{\hat c}\,1\,$;
\item given  $w \in \cal F$ with $|w|>1$, it can be uniquely factorized as  $w=u\, v$, where $u$ and $v$  are non-empty palindrome words. Moreover if $w=w_{\frac {p}{q}} = w_{\frac {p'}{q'}} \, w_{\frac {p''}{q''}}$, then $|u|=p''+q''$ and $|v|=p'+q'$.
\end{enumerate}
\end{theorem}
\begin{proof} The first assertion follows from the definition, whereas  the third easily follows from the second. Let us then prove 2. We proceed by induction on the depth.  For the root node $\frac 1 1$ we get $c=\epsilon$, the empty word, so that the assertion is trivial. Suppose it is true up to depth $n> 1$, and consider 
$\gamma\in \cal T$ with depth$(\gamma)=n$. We have $w_{\gamma}=0 \,c\, 1$ with $c=\tilde c$. On the other hand $\gamma$ is obtained as the child of a left and right parent, say $\alpha$ and $\beta$, one of depth $n-1$ and the other of depth $n-k$, for some $k=2, \dots , n$ (the case in which one parent is an ancestor is left to the reader). Set $w_\alpha=0\, a\, 1$ and $w_\beta=0\, b\, 1$, with $a=\tilde a$ and $b=\tilde b$. Therefore $c=a\, 1\, 0 \, b = {\tilde b}\, 0\, 1\, \tilde a$.
Now consider a child $\delta$ of $\gamma$. If $\delta$ is the right child then by construction $w_\delta = 0\, c\, 1\, 0\, b\, 1=0\, a\, 1\,0\, b\, 1\, 0\, b\, 1=0\, d\, 1$ with $d=a\, 1\,0\, b\, 1\, 0\, b={\tilde b}\, 0\, 1\, {\tilde a}\, 1\, 0\, b$, which is clearly palindromic. If $\delta$ is the left child, the same argument yields $w_\delta = 0\, d'\, 1$ with $d'=a\, 1\, 0\, {\tilde b}\, 0\, 1\, {\tilde a}$. 

\noindent
To show the last statement, we note that from the above it follows that for $w=0\, c\,1 \in \cal F$, the palindrome $c$ has always the structure $c=a\, 1\, 0 \, b = {\tilde b}\, 0\, 1\, \tilde a$, with $a=\tilde a$ and $b=\tilde b$. Therefore we can write $w=u\, v$ with $u=0\,{\tilde b}\, 0$ and $v=1\, \tilde a\, 1$, which are both palindrome words. As for the uniqueness, let $w = uv = ts$ with $u,v,t,s$ all palindromes. Assume without loss that $|u| > |t|$, so $u = th$ and $hv = s$, with $h\neq \epsilon$. Since they are all palindromes, we have $vu = st$, so that $vth = hvt$. Then it readily follows that $w = h^k$ for some positive $k \in \N$. But this is absurd, since it should be $|w|_0 =k|h|_0$ and $|w|_1 =k|h|_1$, but we already know that $|w|_0 =p$ and $|w|_1 =q$ with $p$ and $q$ coprime, and the case $k=1$ would imply $w=u=s=h$, absurd since $|w|>1$ and it couldn't be palindromic. This holds true for each $w \in \cal F$, except for the leftmost and rightmost nodes at each level, for which the uniqueness of the factorization is trivial since $w = 0...01$ or $w = 01...1$.
\end{proof}

\begin{remark}\label{lessico}
The last statement of the above theorem yields two factorizations for $w \in \cal F$ with $|w|>1$: the {\sl palindromic factorization} $w=u\, v$, with $u$ and $v$ both palindromes, and the so called {\sl standard factorization}  $w=w_{\frac {p}{q}} = w_{\frac {p'}{q'}} \, w_{\frac {p''}{q''}}$, in terms of FC sub-words. 
Both of them are unique. 
\end{remark}

\noindent
\begin{remark} It follows from the definition that given a word with standard factorization $w = uv$, with $w_{\frac{p'}{q'}} = u$ and $w_{\frac{p''}{q''}} = v$, then $u(uv)$ and $(uv)v$ are FC words; in particular they are the children of $w$ with the indicated standard factorization. Moreover, if
$|w| \geq  3$, then either $u$ is a proper prefix of $v$, and $v = uv'$ is the standard factorization of $v$, or $v$ is a proper suffix of $u$, in which case $u = u'v$.
\end{remark}

\noindent
Some rather immediate consequences of the above properties are formulated in the following corollaries (see also \cite{BLRS}). 
\begin{corollary}\label{bellu} Let $w=0\, c\, 1$ be a FC word associated to some element of $\cal T$. The FC words associated to its left and right children are given by
$$
0\, (0c)^-\, 1=0\, (c\, 0)^+\, 1\quad \hbox{and} \quad 0\, (1c)^-\, 1=0\, (c1)^+\, 1
$$ 
where $u^-$ and $u^+$ are the shortest palindrome with suffix, respectively prefix, given by $u$.\end{corollary}
\begin{corollary} Let $w=0\, c\, 1$ be a FC word associated to some element of $\cal T$. The maximum among all its
cyclic permutations is realized by the word $ {\tilde w}=1\, c \, 0$.
\end{corollary}

\begin{corollary} The number of FC words of length $n$ is given by Euler totient function $\varphi (n)= |\{0<i<  n : {\rm gcd}\, (i,n)=1\}|$.
\end{corollary}
\begin{proof} From Theorem \ref{propo} we have that $\left|w\right|_1 = p, \left|w\right|_0 = q$. The totient function gives us the number of distinct $p$ which are relatively prime with $n$, which coincides with the number of possible pairs $(p, q = n-p)$ which are relatively prime.
\end{proof} 
\vspace{5mm}

\begin{figure}[htbp]
\begin{center}
%\fbox{\includegraphics[width=\textwidth, trim = 4.5cm 19cm 3.5cm 4cm]{capitolo1/figure/Christoffel tree.pdf}}
\includegraphics[width=\textwidth, trim = 4.4cm 19cm 3.6cm 4cm, clip]{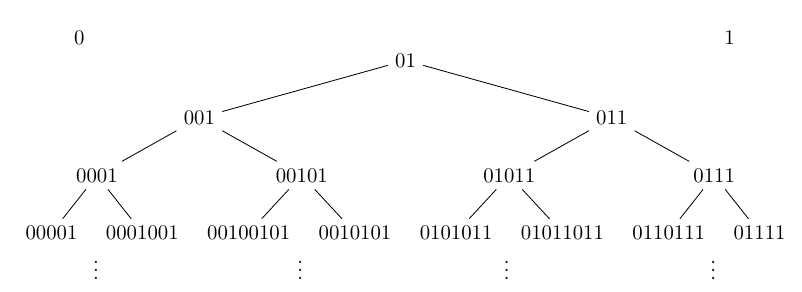}
\caption{The first four level of the Farey-Christoffel words tree }
\label{fig: The Farey-Christoffel tree}
\end{center}
\end{figure}

\section{Relation with cutting and sturmian sequences}
\label{sec:cutting}
Now, given $w\in \cal F$ we call ${|w|_1}/{|w|_0}$ the {\sl slope} of $w$. This is motivated by the following facts. To a given binary word  $w=u_1\cdots u_n$ we can associate a stepwise walk on the lattice $\Z^2$ constructed by moving by a vertical step upwards (resp. horizontal step oriented on the right) for each occurrence of the symbol $1$ (resp. $0$). Clearly, the walks corresponding to $w=0\,c\,1$ and ${\tilde w}=1\, c\,0$ meet at the origin $(0,0)$ and at the point $(|w|_0, |w|_1)$. 
Moreover, letting $\alpha={|w|_1}/{|w|_0}$, the central sequence $c$ is nothing but the {\sl cutting sequence} of the ray having slope $\alpha$, where one writes $0$ each time the ray cuts a vertical line, and $1$ each time it cuts a horizontal line, on the open interval $(0, |w|_0)$. 

\noindent
By the way, the FC word of slope $p/q$ can be defined from the very beginning as a sequence of unitary steps joining points of integer lattice from $(0, 0)$ to $(q, p)$ so that
(i) the corresponding path is the nearest path below the line segment joining these two points; (ii) there are no points of the integer lattice between the path and line segment (see \cite{BLRS}). When the slope is irrational, a similar definition leads to the notion of (infinite) {\sl Sturmian sequence}. 

\noindent
In  Figure \ref{fig:cuttingseq} we report the case with slope $3/5$ (with $r(w)\equiv \tilde w$).
\begin{figure}[h]
\begin{center}
\includegraphics[width=11.0cm, trim = 1mm 11mm 1mm 9mm, clip]{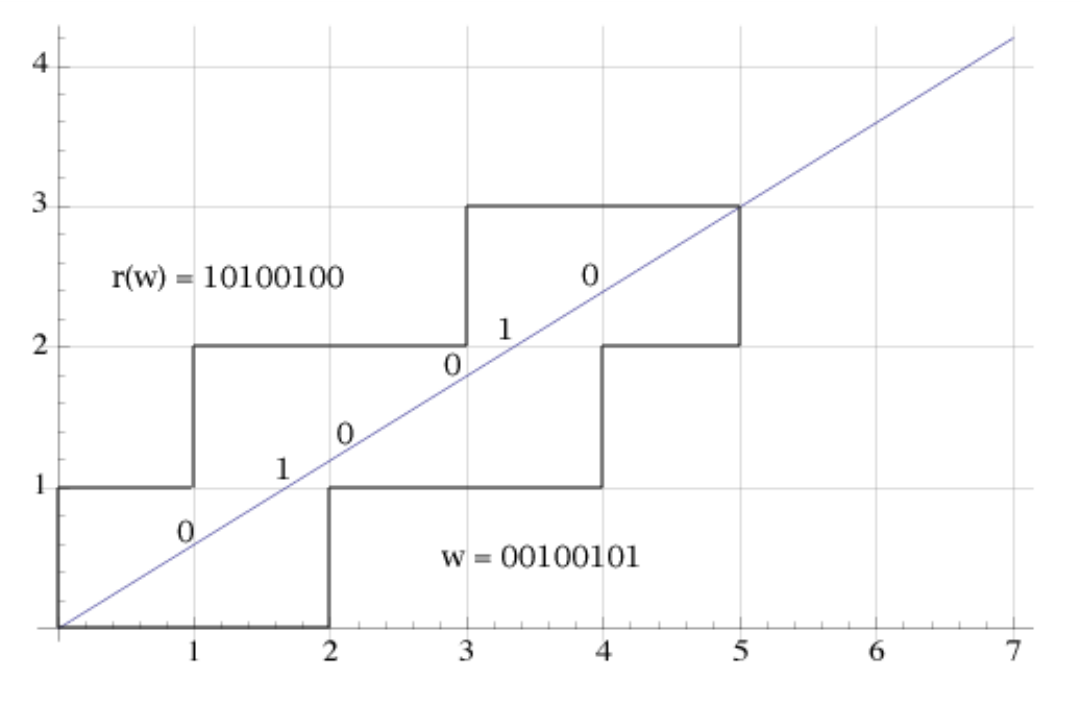}
\caption{}
\label{fig:cuttingseq}
\end{center}
\end{figure}

\noindent
Figure \ref{fig:cuttingparents} shows the cutting sequences of the two parents of $3/5$, namely $1/2$ and $2/3$ (when concatenating two finite cutting sequences, one has to interpose the word $10$, which corresponds to a cut with a corner).
\begin{figure}[h!]
\begin{center}
\includegraphics[width=11.0cm, trim = 1mm 5mm 1mm 17mm, clip]{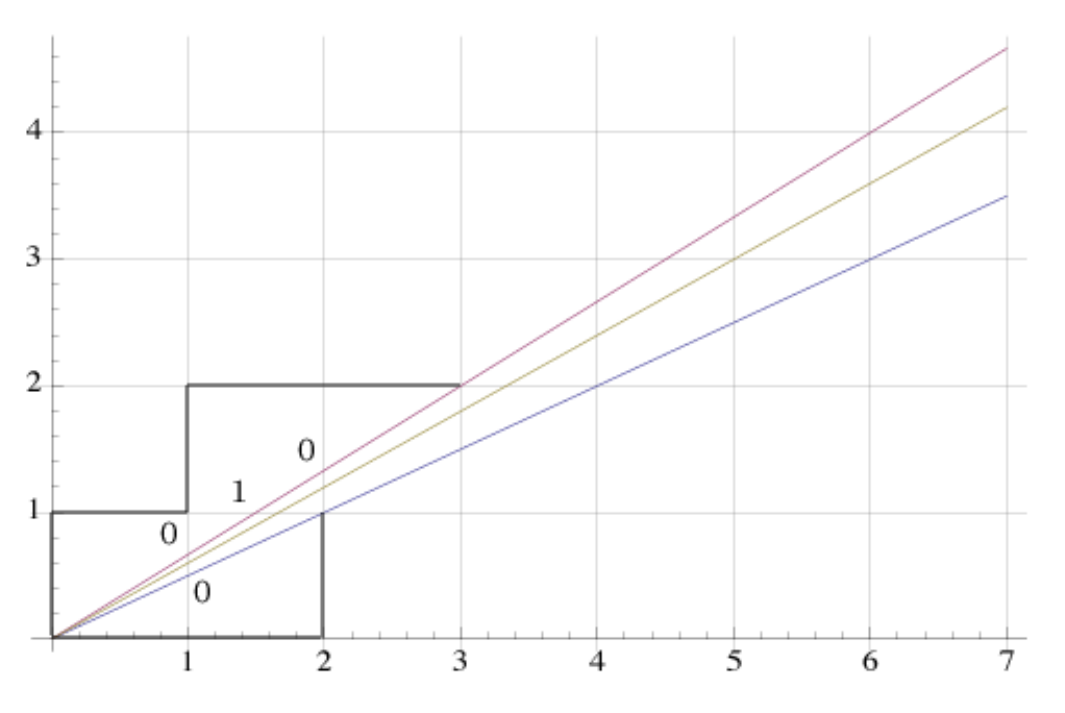}
\caption{}
\label{fig:cuttingparents}
\end{center}
\end{figure}

\begin{remark}\label{factorization}  The standard factorization  $w=w_{\frac {p}{q}} = w_{\frac {p'}{q'}} \, w_{\frac {p''}{q''}}$ in terms of FC sub-words (cf. Remark \ref{lessico}), 
can be obtained geometrically by cutting the walk corresponding to $w$ at the lattice point $(q',p')$ closest to the segment joining $(0,0)$ with $(q,p)$. The last property implies that $pq'-qp'=1$ and therefore $p(p'+q')=p'(p+q)+1=1\, ({\rm mod}\, p+q)$. In the same way, we can show that $q(p''+q'')=1\, ({\rm mod}\, p+q)$. We therefore see that the lengths of the factors $|w_{\frac {p'}{q'}}|=p'+q'$ and  $|w_{\frac {p''}{q''}}|=p''+q''$ are the respective multiplicative inverses in $\{0,1, \dots, p+q-1\}$ of $p$ and $q$. 
\end{remark}

\noindent
Now, putting together Remark \ref{subtree1} and, e.g., \cite{BC}, Section 1 (or else \cite{Py}, Chap. 6), one sees that the FC word $w\equiv w_\alpha$ can also be characterized as the symbolic representation of the orbit $\{R^k_\beta (0)\}_{k=0}^{n-1}$ w.r.t. the partition $S^1=[0,1-\beta)\cup [1-\beta, 1)$, with $n=|w|$ and  $R_\beta: S^1 \to S^1$ the rotation of angle $\beta =\phi (\alpha)$, sometimes also called  the {\sl Sturm sequence of} $\beta$. 
More specifically, set
$$
\epsilon (x) =  \left\{
  \begin{array}{ll} 
 0\; , & \;   0 \leq x < 1-\beta\\ [0.3cm]
 1 \; , & \;  1-\beta \leq x < 1 
 \end{array} \right.
$$
and note that $x+\beta = R_\beta (x) +\epsilon (x)$, which can be iterated to give
$$
x+n\beta = R^n_\beta (x) +\epsilon (R^{n-1}_\beta (x) )+\epsilon (R^{n-2}_\beta (x) )+\cdots + \epsilon (x)=R^n_\beta (x)+[n\beta]
$$
Setting $w=u_1\cdots u_n$, we then have 
\be
\label{sturm}
u_k=\epsilon (R^{k}_\beta (x) )=  [k\beta]-[(k-1)\beta]\quad , \quad k=1, \dots, n.
\ee
Note that, since $\beta \in (0,1)$ we have $u_k\in\{0,1\}$. More precisely, if $\alpha >1$ ($\beta >\frac 12$) in $w$ the symbol $0$ is always isolated and between any two $0$'s there are either $[\alpha]$ or $[\alpha]+1$  $1$'s. If instead $\alpha <1$ ($\beta <\frac 12$) in $w$ the symbol $1$ is  isolated and between any two $1$'s there are either $[1/\alpha]$ or $[1/\alpha]+1$  $0$'s. The opposite plainly happens to ${\hat w}$.

\vsni
\noindent
The above generation rule can be further rephrased as follows (closely mirroring the original construction by Christoffel). Let $p/q \in \cal T$ and set $n=p+q$. Define the group translation $T_{p}: \Z_n \to \Z_n$ as 
$$T_p : x \mapsto x+p \, ({\rm mod}\, n)$$ 

\begin{lemma}\label{lem0} Let $w=u_1\cdots u_n \in \cal F$, with $n>1$,  and $\frac pq = \frac {|w|_1}{|w|_0}$ (so that $|w|=n= p+q$) be the corresponding element of $\cal T$. Consider the partition $\Z_n = Q_0\cup Q_1$ with
$Q_0=\{0,1,\dots, q-1\}$ and $Q_1=\{q,q+1,\dots, n-1\}$. 
$$
u_k = \ell \Longleftrightarrow T_p^{(k-1)}(0) \in Q_\ell, \quad \ell\in \{0,1\}, \quad k=1, \dots ,n
$$
\end{lemma}
\begin{proof}
From the geometric interpretation of the FC words given above, one  deduces the following rule: for any $k=1, \dots ,n$ we have
$u_k = 0$ if $k\cdot p \, ({\rm mod}\, n) > (k-1)\cdot p \, ({\rm mod}\, n)$ and $u_k=1$ in the opposite case. 
\noindent
Now note that, setting $(k-1)\cdot p \, ({\rm mod}\, n) = \ell$, if 
$k\cdot p \, ({\rm mod}\, n) = \ell+p$ then $u_k=0$, whereas if $k\cdot p \, ({\rm mod}\, n) = \ell-q$ then $u_k=1$. In other words, $u_k=0$ if and only if 
$(k-1)\cdot p \, ({\rm mod}\, n) \in Q_0$ and $u_k=1$ if and only if 
$(k-1)\cdot p \, ({\rm mod}\, n) \in Q_1$.
\end{proof}
\begin{remark}\label{subtree1}
If one works with the sub-tree $\cal S$ instead of $\cal T$ (see Remark \ref{subtree}), assigning the initial symbols $0$ and $1$ to $0/1$ and $1/1$ (instead of $1/0$), then the above conclusions are unchanged provided $p/q$ is replaced by $\phi (p/q) = p/(p+q)$ (and $q/p$ by $q/(p+q)$), so that the denominator of the corresponding fraction always equals the length of the FC word. Moreover, the algorithm of Lemma \ref{lem0}, remains unchanged provided we let $T_p$ act on $\Z_q$ instead of $\Z_{p+q}$ and we set $Q_0=\{0,1,\dots, q-p-1\}$ and $Q_1=\{q-p,q-p+1,\dots, q-1\}$. 

\vsni
\noindent
Finally, we note that the map $\phi$ induces the substitution map on FC words given by $0\to 0$ and $1\to 01$. A short reflection shows that this rule can be used to obtain the FC word $w_{\alpha}=u_1\cdots u_n$ constructed above from the Sturm sequence of $\alpha$ itself, that is the word
$w_{\alpha}'= v_1\cdots v_{q}$, with $q=|w|_0$ and
$v_k=[k\alpha]-[(k-1)\alpha]$. 
\end{remark}

\section{Relation with continued fractions}
We have already seen (cf. Lemma \ref{uno}) how the depth of each element $x\in {\cal T}$ is related to the partial quotients of its continued fraction expansion (c.f.e.) $x=[a_0;a_1,\dots ,a_n] $. This connection can be further expanded. One starts by constructing a matrix representation of the positive rationals as follows:
given $z\in \C$ and $X =  \left(
  \begin{array}{cc}
  n & m\\
  t & s
  \end{array}\right) \in SL(2,\Z)$ set $X(z)\coloneqq (nz+m)/(tz+s)$ and identify\vspace{1mm}
 \be\label{ide}
 X\Longleftrightarrow  X(1) =\frac {n+m}{t+s}\in \Q_+
 \ee \vspace{1mm}

\noindent
Clearly $m/s$ and $n/t$ are but the parents of $x$. 
We have
\be\label{AB} {1\over 2} \Longleftrightarrow 
\left(  \begin{array}{cc}
  1&0\\
  1&1 \\
  \end{array}\right)=:A \quad \hbox{e} \quad
   {2\over 1} \Longleftrightarrow 
\left(  \begin{array}{cc}
  1&1\\
  0&1 \\
  \end{array}\right)=:B
\ee 
and moreover
$$
\left(  \begin{array}{cc}
  n&m\\
  t&s \\
  \end{array}\right) \left(  \begin{array}{cc}
  1&0\\
  1&1 \\
  \end{array}\right) 
  =\left(  \begin{array}{cc}
  m+n&m\\
 s+t&s \\
  \end{array}\right) 
 \Longleftrightarrow {m\over s}\oplus {m+n\over s+t}
$$ and
$$ 
\left(  \begin{array}{cc}
  n&m\\
  t&s \\
  \end{array}\right)
  \left(  \begin{array}{cc}
  1&1\\
  0&1 \\
  \end{array}\right)=
  \left(  \begin{array}{cc}
  n&m+n\\
  t&s+t \\
  \end{array}\right)
\Longleftrightarrow {m+n\over
s+t}\oplus {n\over t} 
$$ 
\vskip 0.1cm
\noindent
Hence the matrices $A$ and $B$, when acting from the right, move downwards on $\cal T$, respectively  to the left and to the right. 

\vskip 0.1cm
\noindent
Putting together the above, along with Lemma  \ref{uno}, we get:
\begin{proposition}\label{codaggio}
Each $\frac pq = [a_0;a_1,\dots ,a_n] \in \cal T$, with ${\rm depth} (\frac pq)>1$, corresponds to a unique element
$X\in SL(2,\Z)$, for which there are only two possibilities:
\begin{itemize}
\item $n$ even   $\; \Longrightarrow\; $ $X=B^{a_0}A^{a_1}\cdots A^{a_{n-1}}B^{a_n-1}$
\item $n$ odd $\; \Longrightarrow\; $ $X=B^{a_0}A^{a_1}B^{a_2}\cdots A^{a_n-1}$
\end{itemize} 
Moreover, let $\frac {p}{q}={\frac {p'} {q'}}\oplus {p''\over q''}$ and 
$w_{\frac {p}{q}} = w_{\frac {p'}{q'}} \, w_{\frac {p''}{q''}}$ be the corresponding FC word, then 
$$
X = \left(  \begin{array}{cc}
  |w_{\frac {p''}{q''}}|_1& |w_{\frac {p'}{q'}}|_1\\
   |w_{\frac {p''}{q''}}|_0& |w_{\frac {p'}{q'}}|_0 \\
  \end{array}\right)
 $$
\end{proposition}
\vskip 0.2cm

\noindent
For a given element $x \in \cal T$, the matrix product $X$ can be used to code the descending path which reaches $x$ starting from $\frac 1 1$ as a binary string $\sigma (x) \in \{0,1\}^*$, where each symbol $0$ corresponds to an occurrence of $A$ (down left move) and each symbol $1$ to an occurrence of $B$ (down right move).

\noindent
We may now ask what kind of relation can be established between $\sigma (x)$ and its FC word $w (x)\in \cal F$ (a reverse relation yielding the c.f.e. of $x$ from the corresponding FC word $w$  is discussed in Section \ref{dw} below).
\noindent
The sought relation can be readily obtained from Corollary \ref{bellu}. Indeed, given a palindromic word $u\in\{0,1\}^*$ and a symbol $a\in \{0,1\}$, we  set
\be\label{palcl}
\Phi_a (u) = (u\, a )^+ = ( a\, u)^-
\ee
For example we have $\Phi_0(0110)=01100110$ and $\Phi_1(0110)=011010110$. Note moreover that $\Phi_a(\epsilon)=a$. A direct consequence of Corollary \ref{bellu} is now the following rule.

\begin{proposition}\label{ricodaggio} Let $\sigma (x)=\sigma_1\cdots \sigma_k \in \{0,1\}^*$ be the path of $x \in \cal T$, and $w(x)=0\, c\, 1$ its FC word.
Then we have
\be\label{ricodice}
c = \Phi_{\sigma_k} \circ \Phi_{\sigma_{k-1}} \circ \cdots \circ \Phi_{\sigma_1} (\epsilon)
\ee
\end{proposition}
\noindent
{\sc Example.}  Taking $x=3/5=[0;1,1,2]$, from Proposition \ref{codaggio} we have $\sigma (x) = 010$. Thus, applying  rule (\ref{ricodice}) we get 
$$
c=\Phi_{0} \circ \Phi_{1} \circ \Phi_{0} (\epsilon) = \Phi_{0} \circ \Phi_{1} (0) = \Phi_{0}(010) = 010010.
$$
Finally  $w(x)=0\,c\,1=00100101$ (to be compared with the portions of the trees $\cal T$ and $\cal F$ reproduced above). 
\begin{remark}\label{palin}
The  maps (\ref{palcl}) have been introduced by Aldo de Luca in \cite{DeL}, who called them {\sl palindromic closures}. More generally,
in combinatorial word theory literature the transformation mapping the word $\sigma (x)$ to the central palindrome $c$ of $w(x)$ is usually encoded by a function $Pal : \{0,1\}^* \to \{0,1\}^*$ defined recursively as follows \cite{BdeLR}: set $Pal (\epsilon ) =\epsilon$. If $u=vz \in \{0,1\}^*$ for some $z\in \{0,1\}$ then  $Pal (u) = (Pal (v)z)^+$. Although the two approaches are of course equivalent, the one  outlined above seems more transparently connected to the present construction.
\end{remark}
\subsection{Reversals and duality}\label{dual}

\vskip 0.2cm
\noindent
If we let $A$ and $B$ act on the left we get
$$
\left(  \begin{array}{cc}
  1&0\\
  1&1 \\
  \end{array}\right) 
  \left(  \begin{array}{cc}
  n&m\\
  t&s \\
  \end{array}\right) 
  =\left(  \begin{array}{cc}
  n&m\\
 n+t&m+s \\
  \end{array}\right) 
 \Longleftrightarrow \frac {n+m}{n+m+t+s}
$$ 
and
$$ 
\left(  \begin{array}{cc}
  1&1\\
  0&1 \\
  \end{array}\right)\left(  \begin{array}{cc}
  n&m\\
  t&s \\
  \end{array}\right)
  =
  \left(  \begin{array}{cc}
  n+t&m+s\\
  t&s\\
  \end{array}\right)
\Longleftrightarrow \frac {n+m+t+s}
{s+t}
$$ 
That is, they move a fraction $\frac p q$ respectively to its left and right {\sl descendants} $\frac p {p+q}$ and $\frac {p+q} q$ on $\cal T$. Now, if we associate to a given fraction $x\in \cal T$ a matrix product $X=\prod_{i=1}^d M_i$ where $d = {\rm depth}(x)$, as above, then we can consider the involution $x \to {\hat x}$, where ${\hat x}$ is the rational number represented by the reversed matrix product ${\hat X}=\prod_{i=d}^1 M_i$. This map acts as a permutation on $\Q_+$ and the corresponding permuted tree $\hat {\cal T}$ can be constructed starting from the root node $\frac 1 1$ and writing under each vertex  $\frac p q$ the set of its descendants $\{\frac p {p+q}, \frac {p+q} q\}$.  

\vskip 0.1cm
\noindent
Note moreover that,  according to Proposition \ref{codaggio}, the following rule is in force: let $x=[a_0; a_1,\dots, a_n]$, then 
 \begin{itemize}
\item $n$ even   $\; \Longrightarrow\; $ ${\hat X}=B^{a_n-1}A^{a_{n-1}}\cdots A^{a_1}B^{a_0}$
\item $n$ odd $\; \Longrightarrow\; $ ${\hat X}=A^{a_n-1}B^{a_{n-1}}\cdots A^{a_1}B^{a_0}$
\end{itemize} 
and therefore,  
 \begin{itemize}
\item $n$ even   $\; \Longrightarrow\; $ ${\hat x}=[\, a_n-1\, ; a_{n-1}, \cdots , a_1, a_0+1]$
\item $n$ odd $\; \Longrightarrow\; $ ${\hat x}=[\, 0\, ; a_n-1, a_{n-1}, \cdots , a_1, a_0+1]$
\end{itemize} 
\begin{definition} Let $\sigma (x)=\sigma_1\cdots \sigma_k \in \{0,1\}^*$ be the path of $x \in \cal T$, and $w(x)=0\, c\, 1$ its FC word.
The FC word ${\hat w} = 0\, {\hat c} \, 1$ associated to ${\hat x}$, for which 
$$
{\hat c} = \Phi_{\sigma_1} \circ \Phi_{\sigma_{1}} \circ \cdots \circ \Phi_{\sigma_k} (\epsilon)
$$
is called the {\sl dual word} to $w$. In the same vein, $x$ and ${\hat x}$ will be referred to as {\sl dual elements} in $\cal T$.
\end{definition}
 It turns out (see \cite{BdeLR}) that whenever $w$ and $w^*$ are dual words associated to the irreducible fractions $x=\frac p q$ and ${\hat x}=\frac {\hat p} {\hat q}$, we have $p+q={\hat p}+{\hat q}$ and  ${\hat p}$ and ${\hat q}$ are the respective multiplicative inverses in $\{0,1, \dots, p+q-1\}$ of $p$ and $q$, that is $p{\hat p}, q{\hat q} \equiv 1\, ({\rm mod} \, n)$ with $n=p+q$ (these inverses exist because $p$ and $q$ are relatively prime and therefore are also relatively prime to $n=p+q$. Therefore ${\hat p}$ and ${\hat q}$ are relatively prime). A straightforward consequence of  this property and the content of Remark \ref{factorization} is the following:
\begin{lemma} Let  $x=\frac p q$ and ${\hat x}=\frac {\hat p} {{\hat q}}$ be dual elements in $\cal T$. Then
$$
\frac {p}{q}={\frac {p'} {q'}}\oplus {p''\over q''} \quad \hbox{\rm if and only if} \quad \frac {\hat p}{\hat q}={\frac {p'} {p''}}\oplus {q'\over q''} 
$$
\end{lemma}

\subsection{Motions on $\hat {\cal T}$ and $\hat {\cal F}$.} \label{motions}
We start recalling some results discussed in \cite{BI} about dynamics on  $\hat {\cal T}$. We start observing  that the descendants of a fraction $\frac p q$ are just its pre-images w.r.t. the map
$F:\R_+\to \R_+$ given by \be\label{effe}
F: x\mapsto  \left\{
  \begin{array}{ll} {\displaystyle \frac{x}{1-x}}\; , & \;   0 \leq x \leq 1\\ [0.3cm]
 x-1 \; , & \; x>1 \end{array} \right.
 \ee 
 The map $F$ can thus be used to generate ``vertically'' the permuted tree $\hat {\cal T}$.
Moreover, according to (\cite{BI}, Proposition 2.3), $\hat {\cal T}$ can  also be generated ``horizontally'' by 
means of the map $R:\R_+ \to\R_+$ given by $R(0)=1$, $R(\infty)=0$ and
\be\label{zagier}
R(x)= \frac{1}{ 1-x+2[x]}, \qquad x\in \R_+
\ee
More precisely,  denoting with $r_n$ the $n$-th rational number obtained by `reading' $\cal T$ row by row, from left to right, starting from the root, and letting $r_{n^*}$ be the element of the permuted tree $\hat {\cal T}$ corresponding to $r_{n} \in \cal T$, 
it holds $r_{\hat n}=R^{n-1}(1)$ (or else $r_{n}=R^{{\hat n}-1}(1)$).
\vskip 0.1cm
\noindent
Turning now to consider the permuted FC tree $\hat {\cal F}$, an easy consequence of the construction outlined above (see also \cite{BLRS}, Lemma 2.2) is the following:
\begin{lemma}\label{substi} Let $w$ be the FC word associated to some element $\frac p q \in \cal T$. The FC words associated to its descendants 
$\frac p {p+q}$ and $\frac {p+q} q$ are obtained by applying to $w$ the substitution rules:
$$
 \begin{array}{ll} 
 \displaystyle  {S_0 \, : \,(0,1)\to (0, 01)} \\ [0.3cm]
 \displaystyle  {S_1 \, : \,(0,1)\to (01, 1)} 
 \end{array}
 $$
\end{lemma}
\vskip5mm
\noindent
Now note that any FC word $w$ of length $n$ can be written in the form
\be
\label{isolated1}
w= 0^{n_1}1\, 0^{n_2}\cdots  0^{n_p}\, 1 ,\quad n_i\geq 1\, , \quad \sum_{i=1}^p n_i= q
\ee
whenever its slope $|w|_1/|w_0| = p/q \in (0,1)$, or else
\be
\label{isolated0}
w= 0\, 1^{n_1}\, 0\, 1^{n_2}\cdots  0\, 1^{n_q} ,\quad n_i\geq 1\, , \quad \sum_{i=1}^q n_i= p
\ee
whenever $p/q  >1$. As noted before (cf.remark after eq. \eqref{sturm}, see also \cite{Se}) the integers $n_i$ may get only two values. 
They are $[q/p]$ or $[q/p]+1$, if the slope $p/q$ is smaller than one; $[p/q]$ or $[p/q]+1$, otherwise. Following \cite{Se}, we call the exponent $[q/p]\geq 1$ (or $[p/q]$) the {\sl value} of $w$.

\noindent
This naturally  induces a decomposition of ${\cal F}$ (or $\hat {\cal F}$) as ${\cal F} = {\cal F}_{<1} \cup {\cal F}_{\geq 1}$ (with obvious meaning of the notations), so that $S_0: {\cal F} \to  {\cal F}_{<1}$ and $S_1: {\cal F} \to  {\cal F}_{\geq 1}$, in particular $F_{<1}$ consists of all the left nodes of $\hat {\cal F}$, while ${\cal F}_{\geq 1}$ consists of all the right node, plus the root.

\noindent
We are now ready to introduce a map $T$ on words which generates the ``horizontal'' motion on $\hat {\cal F}$, namely the displacement row by row, from left to right, starting from the root, in a similar way to how $R$ does it for $\hat {\cal T}$.
\vskip 0.5cm 

\begin{theorem}\label{fm} The map  $T$ that moves from a given word $w\in \hat {\cal F}$ to the next one, can be written as $T=T_0\cup T_1$, where the maps $T_0: {\cal F}_{<1} \to  {\cal F}_{\geq 1}$ and $T_1:  {\cal F}_{\geq 1} \to {\cal F}_{<1}$ act as follows:
$$
 \begin{array}{ll} 
 \displaystyle  {T_0 \, : \,(0^{k+1}1, 0^k1)\to ((01)^k1, (01)^{k-1}1}) \\ [0.3cm]
 \displaystyle  {T_1 \, : \,(01^{k+1},01^k)\to (0^k1, 0^{k+1}1)} 
 \end{array}
 $$
 where $k$ is the value of $w$.
\end{theorem}
\begin{proof}
Let $w= 0^{n_1}1\, 0^{n_2}\cdots  0^{n_p}$ with:
$$
n_i = k \,\, \mbox {or} \,\, k+1 \quad \mbox{for} \,\, i=1, \ldots, p \, \,  , \quad \mbox{and} \,\,\sum_{i=1}^p n_i= q \, .
$$
Let $w'$ be the parent node of $w$ and $T(w)$, we have that $w'$ is given by $S_0^{-1} (w)$ and, recalling that $0^0 = \epsilon$, we have:
$$
w' = 0^{n_1 -1}10^{n_2 -1}1\ldots0^{n_p -1}1 \, .
$$
Then, thanks to $S_1$, we have 
$$
T(w) = S_1 (w') = (01)^{n_1 -1}1(01)^{n_2 -1}1\ldots(01)^{n_p -1}1 \, ,
$$
and we have shown $T_0 = T\restrict[{\mathcal{F}}_{<1}]$. \vspace{1mm} \\
Now we will show that $T_1 = T\restrict[{\mathcal{F}}_{\geq1 }]$ by induction on the depth $m$ of the word $w$. For $m=1$, that $T(01) = T_1 (01) = 001$ is trivial. Let's then assume it holds true for each $w$ at depth $m$, and we will prove it for $m+1$. 
Let $w = 01^{n_1}01^{n_2}\ldots01^{n_q}$ with:
$$
n_i = k \,\, \mbox {or} \,\, k+1 \quad \mbox{for} \,\, i=1, \ldots, q  \, \,  , \quad \mbox{and} \, \, \sum_{i=1}^q n_i= p \, .
$$
Let  $w'$ be the parent node of $w$, and $w'' = T (w')$ the parent node of $T(w)$. Then $T(w) = S_0 (w'')$. Clearly, $w'$ is given by
$$
w' = S_1^{-1} (w) = 01^{n_1 -1}01^{n_2 -1}\ldots01^{n_q -1} .
$$
Now, let us consider the $q$ subwords $01^{n_i-1}$ individually, and we call $\overline{n}_i$ the complement of $n_i$ in the set $\{ k, k+1\}$. Then, if $k>1$, we have, by the induction hypothesis, that $w'' = T_1(w')$ and so, by the action of $T_1$, the subword $01^{n_i-1}$ becomes $0^{\overline{n}_i -1}1$, and applying $S_0$, we get:
$$
T(w) = S_0 (w'') = 0^{\overline{n}_1}10^{\overline{n}_2}1\ldots0^{\overline{n}_q}1
$$
which we wanted to show. \\
On the other hand, if $k=1$, then the subword $01^{n_i -1}$ is either $0$ or $01$, so that $w' \in {\mathcal{F}}_{<1}$ and $T(w')=T_0(w')$. Thus, applying $T_0$, it is clear\footnote{The definition of $T_0$ given by the theorem is equivalent to saying that for each subword $0^n1$ we substitute each of the first $n-1$  zeros with $01$, while what remains, i.e. $01$, we substitute with $1$.} that $\forall \, i=1,\ldots,q$ for which $n_i -1 = 0$, we get $01$, while $\forall \, i=1,\ldots,q$ for which $n_i -1 = 1$, we get $1$. And, applying $S_0$, we get that $01$ becomes $001$, while $1$ become $01$. So, putting it all together, we have
$$
01^{n_i} \xlongrightarrow{S^{-1}} 01^{n_i -1} \xlongrightarrow{T} 0^{\overline{n}_i -1}1 \xlongrightarrow{S_0} 0^{\overline{n}_i}1
$$
which is what we needed to prove.
\end{proof} 
\vskip 0.2cm
\noindent
The map $T$, defined for FC words, can be used to generate ``horizontally" the tree $\hat {\cal F}$ as the map $R$ can be used to generate ``horizontally" the tree $\hat {\cal T}$. Since $R$ is defined on $\R_+$, we would like to find an extension of $T$ such that the correspondence with $R$ is not limited to $\Q_+$. \\
To this end, let us first recall the definition and characterization of a notion already introduced in Section \ref{sec:cutting}.
As described by Aldo de Luca and Filippo Mignosi in \cite{deLM}:
\begin{quote}
A Sturmian word\footnote{In this paper we  use the term ``sequence".} can be characterized as a (one-sided) infinite word which is not ultimately periodic and is such that for any positive integer $n$ the number $g(n)$ of its factors of length $n$ is minimal (i.e. $g(n)=n+ 1$). A Sturmian word can also be defined by considering the intersections with a squared-lattice of a semi-line having a slope which is an irrational number\footnote{This construction is usually called \emph{billiard sequence}. We will limit ourself to consider semi-line with intercept $0$, i.e.: starting at the origin $(0,0)$.}.
\end{quote}
Another common characterization of Sturmian sequences is the following: an aperiodic sequence over a binary alphabet is Sturmian if and only if it is \emph{balanced} (see \cite{BS}, \cite{HM}). An infinite word $w$ on $\{0,1\}$ is balanced if given two factors of $w$, $u$ and $v$, with $|u| = |v|$, the difference between $|u|_0$ and $|v|_0$, or equivalently between $|u|_1$ and $|v|_1$, is at most $1$. \\
We recall that Sturmian sequences can also be regarded as infinite cutting sequences (cf. Section 2), thus enjoying the property that if the slope $x$ is $>1$ then they have isolated $0$'s interspersed with blocks of the form $1^k$ or $1^{k+1}$ ($k = \lfloor 1/x \rfloor$), or, otherwise, they have isolated $1$'s, with blocks of the form $0^k$ or $0^{k+1}$ if $x<1$ ($k=\lfloor x \rfloor$) \cite{Se}. We can now state the following:
\begin{theorem}
\label{extensionT}
Given a Sturmian sequence $w$ with irrational slope $x$ (and intercept $0$), the sequence $\overline{w}$ given by $0\overline{w} = T(0w)$ is a Sturmian sequence. Moreover, its slope is $R(x)$. 
\end{theorem}
We consider, in this theorem, Sturmian sequences preceded by a $0$ in the same way we consider, in Theorem \ref{fm}, FC words in the form $0c1$ with $c$ finite cutting sequence. In this way, without further adjustments, the map $T$ in Theorem \ref{fm} is well defined on the set of Sturmian sequences with irrational slope (and intercept $0$).
To prove this theorem, we first show that $T(w)$ is a balanced sequence, and we will do so through two lemmas.
\begin{lemma}
Given $T_1 : (01^{k+1}, 01^k) \rightarrow (0^k1, 0^{k+1}1)$ and a Sturmian sequence $w$ with irrational slope $x > 1$ (and intercept $0$), then $\overline{w}$ given by $T_1(0w) = 0\overline{w}$ is balanced.
\end{lemma}
\begin{proof}
We will use induction on the length $n$ of the factors of $\overline{w}$.
For $n=1$, it is trivial that the difference in the number of $0$'s between two factors is at most $1$. Moreover, the statement clearly holds for $1 \leq n \leq k+1$, since there can only be at most one $1$ in each factor.\\
Let the statement be true for some $n>k+1$, and let's assume, by contradiction, that there exist two factors $\overline{u}$ and $\overline{v}$ with $|\overline{u}| = |\overline{v}| = n+1$ and\footnote{Without loss of generality, we may assume equality, instead of  $|\overline{u}|_1 \geq |\overline{v}|_1 + 2$, since the case $|\overline{u}|_1 > |\overline{v}|_1 + 2$ immediately contradicts the inductive hypothesis.} $|\overline{u}|_1 = |\overline{v}|_1 + 2$.
Then it follows that $\overline{u}$ and $\overline{v}$ are of the form $\overline{u}=1\overline{u}'1$ and $\overline{v}=0\overline{v}'0$; that is, the ends of the two words must necessarily be different. Otherwise, by considering the subwords obtained by removing an equal symbol at the ends\footnote{Clearly, the opposite situation, $\overline{u}=0\overline{u}'0$ and $\overline{v} = 1\overline{v}'1$, would be even worse.} we would obtain words of length $n$ that differ in the number of $1$'s by two, contradicting the inductive hypothesis.
We can thus consider the factor obtained by extending\footnote{This is always possible thanks to the definition of $T_1$ and the characteristics of $w$.} the block of $0$'s that $\overline{v}$ has as a prefix and the block of $0$'s that it has as a suffix, obtaining $0^t\overline{v}'0^s1$ for some $t,s \leq k$. Comparing it with $\overline{u}'1$, these two words do not have the same length, but they certainly have the same number of $1$'s and, therefore, the same number of blocks, either $0^k1$ or $0^{k+1}1$. Since we have added at least a $1$ to $\overline{v}$ and removed a $1$ from $\overline{u}$, it follows that $|0^t\overline{v}'0^s1| \geq |\overline{u}'1| + 2$.
Denoting by $a$ and $b$ respectively the number of $0^{k+1}1$ blocks in $\overline{u}'1$ and in $0^t\overline{v}'0^s1$, we have $b \geq a + 2$.\\
Considering the preimages via $T_1$, we obtain two subwords of $w$, which we denote by $T_1^{-1} (\overline{u}'1) = u$ and $T_1^{-1} (0^t\overline{v}'0^s1) = v$, which have the same number $d$ of $01\ldots1$ blocks. However, $u$ has $a$ blocks of type $01^{k}$, whereas $v$  has $b$; consequently, $u$ has $d - a$ block of type $01^{k+1}$, whereas $v$ has $d - b$.
This implies that $|u| \geq |v| + 2$, with the same number of $0$'s. Then, by removing the prefix $0$ from $u = 0u'$ and appending to $v$, as suffix, the symbol $0$ that follows it, we obtain $u'$ and $v0$, two factors of $w$, with $|u'| \geq |v0|$ and $|v0|_0 - |u'|_0 = 2$, which is absurd because it contradicts the hypothesis that $w$ is a Sturmian sequence and, as such, should be balanced.
\end{proof}

\begin{lemma}
Given $T_0 : (0^{k+1}1, 0^k1) \rightarrow ((01)^k1, (01)^{k-1}1)$ and a Sturmian sequence $w$ with irrational slope $ x < 1$, (and intercept $0$), then $\overline{w}$ given by $T_0(0w) = 0\overline{w}$ is balanced.
\end{lemma}
\begin{proof} 
We divide the proof into two parts, and in both cases, as in the previous proof, we proceed by induction on the length $n$ of factors of $\overline{w}$.\vskip2mm
\noindent
\emph{First case}: $\lfloor x \rfloor = 1$; that is, $k=1$ and $w$ is of the form\footnote{The $\scalebox{1.5}{+\!\!\!+}$ symbol, used for list concatenation in Haskell, is used here, with an abuse of notation, for infinite concatenations like the symbol $\sum$ would be used. } ${\scalebox{1.5}{+\!\!\!+}}_{i\in\N} (0^{s_i}1)_i$, with $s_i = 1$  or $2$.
We can observe that, for the $w$ under consideration,  $T_0 : (001, 01) \rightarrow (011, 1)$. Then, for $n= 1, \, 2$ and $3$ it is trivial that the difference in the number of $0$'s between two factors is at most $1$.\\
Assume that the statement holds for some $n>3$, and let us prove it for $n+1$.\\
Suppose, by contradiction, that it does not hold; that is\footnote{As in the proof above.}, there exist two factors of the form $1\overline{u}1$ e $0\overline{v}0$ with $|1\overline{u}1|_1 = |0\overline{v}0|_1 + 2$ and $|1\overline{u}1|_0 = |0\overline{v}0|_0 - 2$. We know that each $0$ must be followed by at least two $1$'s, thus we can consider the factors $0\overline{v}011$ and $\overline{u}1$. Hence $|0\overline{v}011|_1= |\overline{u}1|_1 + 1 = a$, and $|0\overline{v}011|_0=|\overline{u}1|_0 + 2 = b$.\\
Considering $T_0$ and the given $w$, we have that, via $T_0^{-1}$, each $(01)$ corresponds to $0$ and all the remaining $1$ corresponds to $01$. 
Then, we get $T_0^{-1}(\overline{u}1) = 0u$ and $T_0^{-1}(0\overline{v}011) = v1$ with $|v1|_0 = | 0u|_0 + 1= a$, $|0u|_1 = a - 1 - (b-2) = a - b + 1$, and $|v1|_1 = a-b$. Hence $|0u| = 2a-b =  |v1|$. Now, considering the two factors $u$ and $v$, we have $|u| = |v|$ with $|u|_0 = |v|_0 - 2$, which is absurd because it contradicts the hypothesis that $w$ is a Sturmian sequence and as such should be balanced.
\vskip2mm
\noindent
\emph{Second case}: $\lfloor x \rfloor \geq 2$; that is, $k \geq 2$ and $w$ is of the form ${\scalebox{1.5}{+\!\!\!+}}_{i\in\N} (0^{s_i}1)_i$, with $s_i = k$  or $k+1$ and $\overline{w} = {\scalebox{1.5}{+\!\!\!+}}_{j\in\N} (01^{t_j})_j$, with $t_j=1$ or $2$, i.e.: it will be a semi-infinite sequence composed of subwords $01$ and $011$.
Then, for $n= 1, \, 2$ and $3$ it is trivial that the difference in hte number of $0$'s between two factors is at most $1$.\\
Assume that the statement holds for some $n>3$, and let us prove it for $n+1$.\\
Suppose, by contradiction, that it does not hold; again, we would have two factors of the form $1\overline{u}1$ e $0\overline{v}0$ with $|1\overline{u}1|_1 = |0\overline{v}0|_1 + 2$ and $|1\overline{u}1|_0 = |0\overline{v}0|_0 - 2$. We then consider the factors $01^t1\overline{u}11^s$, with $t,s= 0$ or $1$, obtained by extending the blocks of $1$'s in the prefix and suffix, and $0\overline{v}01$, so that $|0\overline{v}01|_1 = |01^t1\overline{u}11^s|_1 -1 -t -s= a$ and $|0\overline{v}01|_0 = |01^t1\overline{u}11^s|_0 + 1= b$. \\
Considering $T_0$ and the given $w$, we have that, via $T_0^{-1}$, each $(01)$ corresponds to $0$ and all the remaining $1$ corresponds to $01$. Then, considering $T_0^{-1}(0\overline{v}01)=v$ and $T_0^{-1}(01^t1\overline{u}11^s) = u$, we get $|v|_0 = a$, $|v|_1 = a-b$, $|u|_0 = a +1 + t +s $, and $|u|_1 = a +2 + t +s - b$.  
But, since $0\overline{v}01$ ends in $01$,  whether it is followed by $0$ or by $1$, we have that $v$ is always followed by another $0$. Thus, $|u|_0 - |v0|_0 = t + s$ and $|u| = |v0| + 2 +2t +2s$.
We now have four cases:
\begin{enumerate}
\item if $(t,s)=(0,0)$ we have $u = T_0^{-1}(01\overline{u}1) = 00u_1$, so that\\ $|v0|_0 - |u_1|_0 = 2$ with $|u_1| = |v0|$;
\item if $(t,s)=(0,1)$ we have $u = T_0^{-1}(01\overline{u}11) = 00u_{2}01$, so that\\ $|v0|_0 - |u_2|_0 = 2$ with $|u_2| = |v0|$;
\item if $(t,s)=(1,0)$, we have $u = T_0^{-1}(011\overline{u}1) = 0010u_3$, so that\\ $|v0|_0 - |u_2|_0 = 2$ with $|u_3| = |v0|$;
\item if $(t,s)=(1,1)$, we have $u = T_0^{-1}(011\overline{u}11) = 00100u_4$, so that\\ $|v0|_0 - |u_4|_0 = 2$ with $|u_4| = |v0| + 1$.
\end{enumerate}
All four results are absurd, since the hypothesis states that $w$ is a Sturmian sequence and, as such,  balanced.
\end{proof}

\vspace{0.5cm}
\noindent
Now we can finally prove the Theorem \ref{extensionT}.\\

%\vspace{0.01cm}
\noindent
\begin{proof}
When considering $T_0$, we have that the slope of $w$ is $x<1$. We call $a_n$ the number of $1$'s in the first $n$ blocks of $w$, and $b_n$ the number of $0$'s. For each $0$ in $w$ we get a $1$ in $\overline{w}$, and for all $0$'s, except those followed by a $1$, we get a $0$ in $\overline{w}$. That means that the ratio between $1$'s and $0$'s in $\overline{w}$ is given by:
$$
\frac{b_n}{b_n - a_n} = \frac{1}{\frac{b_n - a_n}{b_n}} =  \frac{1}{1 - \frac{a_n}{b_n}}
$$
and, by considering the limit, we get
$$
\lim_{n \to \infty} \frac{1}{1 - \frac{a_n}{b_n}} = \frac{1}{1 - x} = R(x)
$$
On the other hand, considering $T_1$, we have that the slope of $w$ is $x>1$ and the value $k=\lfloor x \rfloor$. In the first $n$ blocks of $w$, we have exactly $n$ $0$'s, and we have $p$ times $k$ $1$'s, and $q$ times $k+1$ $1$'s, with $p+q=n$. 
For each $k$ block of $1$'s in $w$, we get $k+1$ $0$'s in $\overline{w}$, and for each $k+1$ block of $1$'s, we get $k$ $0$'s, while for each block of any kind in $w$ we get exactly one $1$ in $\overline{w}$. Thus, the ratio between $1$'s and $0$'s in $\overline{w}$ is given by:
$$
\frac{n}{p(k+1) + q(k)}=\frac{n}{p(\lfloor x \rfloor+1) + q(\lfloor x \rfloor)} = \frac{1}{\frac{p}{n}(\lfloor x \rfloor+1) + \frac{q}{n}(\lfloor x \rfloor)} = \frac{1}{\lfloor x \rfloor + \frac{p}{n}}
$$
Now, considering that
$$
x = \lim_{n \to \infty} \frac{a_n}{b_n} = \lim_{n \to \infty} \frac{p(k) + q(k+1)}{n} = k + \lim_{n \to \infty} \frac{q}{n}
$$
we have that $\frac{q}{n}$ tends to $\{x\}$, hence $\frac{p}{n}$ tends to $1 - \{x\}$, and we get
$$
\lim_{n \to \infty} \frac{1}{\lfloor x \rfloor + \frac{p}{n}} =  \frac{1}{\lfloor x \rfloor +1 - \{x\}} = \frac{1}{1 - x +2\lfloor x \rfloor} = R(x)
$$
Thus, the ratio between $1$'s and $0$'s in $0\overline{w}=T(0w)$ is irrational; hence the sequence is aperiodic and, since we have shown in the two lemmas above that it is also balanced, it follows that is a Sturmian sequence.
\end{proof}
\remark ({\bf Connection with S-adic systems})
\vskip 0.1cm
\noindent
On the permuted tree $\hat {\cal T}$ one can introduce a symmetric random walk $(Z_k)_{k\geq 1}$ in the following way: set $Z_1=\frac 11$ and if $Z_k=\frac pq$ then either $Z_{k+1}=\frac p{p+q}$ or  $Z_{k+1}=\frac {p+q}{q}$, both with probability $\frac 12$. In \cite{BI} it is proved that this process enters any non empty interval $I=(a,b)\subset \R_+$ almost surely (Thm. 1.12) and, more specifically, it does it with asymptotic frequency $\rho (I)=\int_a^b d\rho (x)$ (Corollary 3.7), where 
$\rho: {\overline \R}_+\to [0,1]$ encodes the infinite path of $x \in {\overline \R}_+$ by interpreting it as the binary expansion of a real number in $[0,1]$. Differently said, $\rho (0)=0$, $\rho (\infty)=1$ and, if $x=[a_0; a_1, a_2, \dots]$, then 
\be\label{rho}
\rho (x)=0\; .\;{\underbrace
{11\dots 1}_{a_0}}\, {\underbrace {00\dots 0}_{a_1}}\;{\underbrace
{11\dots 1}_{a_2}}\; \cdots 
\ee
A similar study can be pursued on the permuted tree $\hat {\cal F}$, starting from the observation that the substitutions $S_0$ and  $S_1$ defined in Lemma \ref{substi}, whose incidence matrices coincide with $A$ and $B$, define a so called S-adic system (see \cite{Qu}, pp. 87-109, and  \cite{BD}), which, however, are rarely considered as generating a random process. For an interesting analysis of the spectral properties of S-adic random system arising from an i.i.d. sequence of unimodular substitutions, see \cite{So}. 
Besides, it would be also interesting to study the dynamics induced by the map $T$ defined in Thm. \ref{fm} from a statistical point of view (see the next Section for some results for the map $R$).

\vskip 0.2cm
\noindent

\remark({\bf FC words and musical scales})

\vskip 0.1cm
\noindent
FC words that are dual to one another deserve an important role in the theory of well-formed scales in music theory \cite {CC} (see also \cite{I}). Loosely speaking, we first say that a scale is {\sl generated} if its elements can be obtained by an iterated application of a generator\footnote{ Western music, since its Greek origins, has primarily used the fifth interval as a generator of harmonic systems.}, i.e. a fixed transposition on a given pitch class, and then we say that a generated scale is {\sl well-formed}, if each generating interval spans the same number of scale steps (including the return to origin interval).
A remarkable property brought into light by the recent developments in music and combinatorics on words \cite {DCN} starts from the observation that, for example, the FC word $w=0001001$, corresponding to the fraction 2/5, is the sequence of intervals corresponding to the ancient mixolydian (descending) mode B'-A-G-F-E-D-C-(B) (or else to the ascending lydian mode as a medieval ecclesiastical mode), where 0 stands for a tone and 1 for a semi-tone. If we now take the slope 4/3, where 4 and 3 are the multiplicative inverses of respectively 2 and 5 modulo 7, the dual FC word ${\hat w}=0101011$ corresponds to the same mode B'-E-A-D-G-C-F-(Bb) but in a different presentation, where now 0 stands for a descending perfect fifth (the generator) and 1 for an ascending perfect fourth (the generator's complement within the octave), so that the pitches reached thereby all lie within the octave under the initial B'. The two presentations are respectively called the {\sl scale-step pattern} and the {\sl scale folding} of the mode. The other seven diatonic modes forming of the diatonic 7-notes family can be obtained from this mode by conjugation, where we say that two elements $w$ and $w'$ of $\{0,1\}^*$ are {\sl conjugate} if there exist words $u$ and $v$ such that $w=uv$ and $w'=vu$ (or equivalently if they are conjugated in the free group $<0,1>$).
\vskip0.2cm
\begin{figure}[h!]
\begin{center}
\includegraphics[width=12.0cm, trim = 1mm 65mm 1mm 75.7mm, clip]{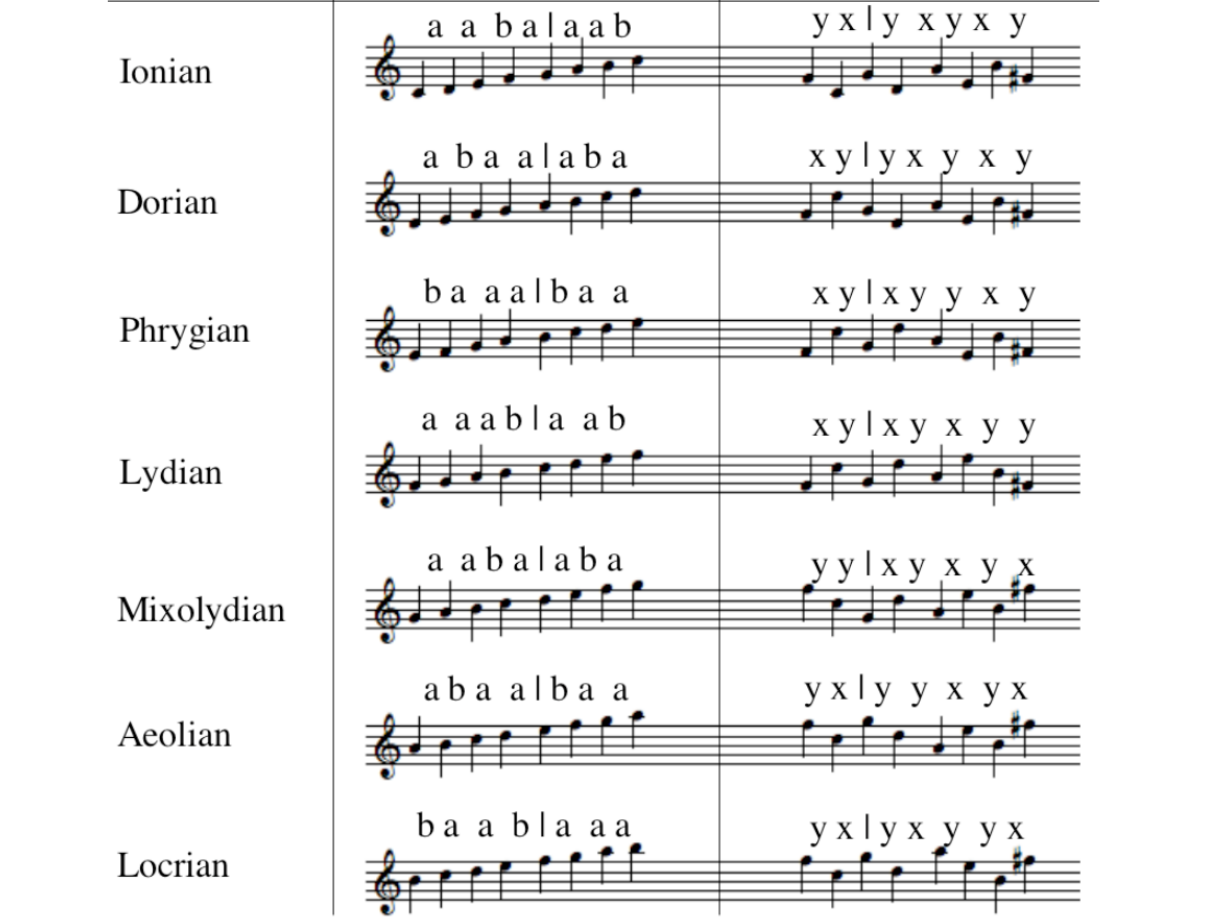}
\caption{}
\label{fig:scales}
\end{center}
\end{figure}
\noindent
Figure \ref{fig:scales} (Figure 8 of Thomas Noll's paper \cite{No}) shows the musical folding of each (ecclesiastical) diatonic mode displayed with their corresponding scale step pattern. In the table, which is an instance of Farey-Christoffel duality, the symbol $a$ stands for a tone, while $b$ for a semi-tone, whereas $x$ is an ascending fifth and $y$ a descending fourth.

\noindent
In the same vein can be treated other musical scales, such as the pentatonic scales (starting from the scale-step pattern 01011, whose dual is 00101), or the so called `tetractys' (starting with 011, which is self-dual). This quick sketch can hopefully give a sense of the richness lying in the folds of the interaction between these domains. One interpretation of this richness may come from thinking of the FC words as divisions into ``almost equal'' parts (cf. section 17.3 in \cite{Re}), in the following sense: if $d < n$ are relatively prime, then $n=dq+r$ with positive remainder $r$. Therefore $n$ is not divisible into $d$ equal \emph{integer} parts. On the other hand,  the second-best solution is to divide $n$ into $d-r$ equal parts of size $q$, and the remaining $r$ parts of size $q+1$. By writing these parts as a word of length $d$, as evenly as possible, one obtains a FC word (cf. the geometric interpretation presented at the beginning of Section \ref{sec:cutting} and in Figure \ref{fig:cuttingseq}).
%%%%%%%%%%%%%%%%%%%%%%%%%%%%%%
\section{Ordering and dynamical systems}\label{dw}
\vskip 0.2cm
\noindent
We shall now discuss some further aspects of the relation between the c.f.e. of a given element of $x\in \cal T$ and its FC word $w\in \cal F$. 
To this end we recall that any FC word $w$ of length $n$ can be written in the form shown in \eqref{isolated1} or \eqref{isolated0} depending on its slope (cf. Section \ref{motions}).\\
Then, we can construct a {\sl derived word} $w'$ via the following algorithm: suppose that the slope $p/q$ of $w$ is smaller than one and its value is $k$ (that is $[q/p]=k$). Then the symbol $1$ is isolated and we perform the substitution
$0\to 0$ and $0^k1 \to 1$. If, instead, the slope $p/q$ is larger than one, and $[p/q]=k$, then the symbol $0$ is isolated and we perform the substitution $1\to 1$ and $01^k \to 0$. We keep iterating this procedure until we end up with a single symbol, $0$ or $1$, while recording the values $a_0, a_1, \dots, a_n$ of the derived sequences\footnote{If the slope of the initial sequence $w$ is smaller than one we set $a_0=0$. On the other hand the value of a single symbol can be taken to be $\infty$ (as it seems natural when passing to infinite sequences by indefinite repetition of the finite string).}. We have the following:
\begin{proposition}\label{codaggio1}
Let $x\in \cal T$ and $w\in \cal F$ be the corresponding FC word. The values of the successively derived words $w', w'', \dots$ coincide with the partial quotients of the c.f.e. of $x$. 
\end{proposition}
\begin{proof}
The proof amounts to noting that the reduction procedure corresponds to repeated applications to the slope of the map \eqref{effe} $F:\R_+\to \R_+$ given by
\begin{equation*}
F: x\mapsto  \left\{
  \begin{array}{ll} {\displaystyle \frac{x}{1-x}}\; , & \;   0 \leq x \leq 1\\ [0.3cm]
 x-1 \; , & \; x>1 \end{array} \right.
 \end{equation*}
 whose action of c.f.e.'s is\footnote{In the first case, if $a_1=1$ one sets $[0;a_1-1,a_2,\dots]=[a_2;a_3,a_4,\dots]$.}
 \be\label{faction}
F: [a_0;a_1, a_2, \dots] \mapsto \left\{
  \begin{array}{ll} [0;a_1-1,a_2,\dots]\; , & \;   a_0=0\\ [0.3cm]
 [a_0-1;a_1,a_2,\dots] \; , & \; a_0>0 \end{array} \right.
 \ee
More precisely, if $w$ has slope $x$ and value $k$ then the derived sequence $w'$ has slope $F^k(x)$, and value either $[F^k(x)]$ or $[1/F^k(x)]$. 
\end{proof}

\vskip 0.5cm
\noindent
{\sc Example.}  For $p/q=3/5=[0;1,1,2]$ and $w=00100101$ we get the following table.
\vskip 0.5cm
\begin{center}
\begin{tabular}{|c|c|c|c|}
\hline
derivation step & FC word & slope & value \\
\hline
0 & $00100101$ & $3/5$  &$1$ \\ 
1 & $01011$ & $3/2$ & $1$ \\ 
2 & $001$ & $1/2$ & $2$ \\ 
3 & $1$ & $1/0$  & $\infty$ \\ 
\hline
\end{tabular}
\end{center}
\vskip 0.5cm

\vskip 0.2cm
\noindent
Now, any $\frac pq\in \cal T$ of depth $d\geq 1$ is the descendant of another fraction $\frac {p'}{q'}\in \cal T$ of depth $d-1$, which we call its {\sl antecedent}, given by the following rule: if $p>q$ then $q'=q$ and $p'=p-q$; if instead $q>p$ then $p'=p$ and $q'=q-p$. Differently said, $\frac {p'}{q'}=F(\frac {p}{q})$.
Therefore, according to what we have said in Section \ref{dual}, the binary coding $\sigma (x)=\sigma_1\cdots \sigma_{k}$ of an element $x\in \cal T$ of depth $k+1$ can be computed in terms of the symbolic orbit of $x$ with the map $F$:
\be\label{F}
\sigma_i(x)= \left\{
  \begin{array}{ll}  0\; , & \;  F^{i-1}(x) \leq 1\, ,\\ [0.3cm]
1 \; , & \; F^{i-1}(x) >1\, , \end{array} \right. \qquad i=1, \dots, k
\ee 
This rule can be immediately checked for the already discussed example $x=3/5$. For a less trivial example consider the fraction $x= 65/19$, whose c.f.e. is $[3;2,2,1,2]$. It has depth $3+2+2+1+2=10$ and from Proposition \ref{codaggio} its symbolic coding is $\sigma (x)=111001101$. 
Without knowing the c.f.e. this binary sequence can be obtained from the antecedents, i.e. the $F$-images of $x$ till the root of $\cal T$. They are
$$
\frac {65} {19} , \quad \frac {46} {19} ,  \quad \frac {27} {19} ,  \quad \frac {8} {19} , \quad \frac {8} {11} , \quad \frac {8} {3} , \quad \frac {5} {3} , \quad 
\frac {2} {3} , \quad \frac {2} {1} , \quad \left(\frac {1} {1} \right)
$$
and one easily checks that the sequence obtained applying rule (\ref{F}) is just $\sigma (x)$ written above.
\vsni
\noindent 
We have said that the tree $\cal T$ enumerates the positive rationals, but what is the ordering induced on $\Q_+$?
Denoting again with $r_n$ the $n$-th rational number obtained by `reading' $\cal T$ row by row, from left to right, starting from the root, we have
$$
r_1=\frac 1 1, \; r_2=\frac 1 2, \; r_3=\frac 2 1, \; r_4=\frac 1 3, \; r_5=\frac 2 3, \; r_6=\frac 3 2, \; r_7=\frac 31, \; r_8=\frac 1 4, \; \cdots
$$
The general rule is in the following:
\begin{theorem}\label{regola}  Given $1\ne  x\in \cal T$, let $\sigma (x)=\sigma_1\cdots \sigma_{k}$ be its binary coding. Then we have $x=r_n$ with $n=2^{k}+\sum_{l=1}^{k} \sigma_l 2^{k-l}$.
\end{theorem} 
\noindent
{\sc Example.} The number $x= 65/19$ yields $n=2^{9}+2^{8}+2^7+2^6+2^3+2^2+2^0=973$, namely $65/19$ is the nine hundred  seventy-third rational number in the Stern-Brocot ordering. 

\vskip 0.2cm
\noindent 
\begin{proof} 
Let $r_{\hat n}$ be the element of the permuted tree $\hat {\cal T}$  corresponding to $r_{n} \in \cal T$ (or else $r_{n}$ and $r_{\hat n}$ are dual elements in $\cal T$). Then $n=2^{k}+\sum_{l=1}^{k} \sigma_l 2^{k-l}$ if and only if ${\hat n}=2^{k}+\sum_{l=1}^{k} \sigma_{l} 2^{l-1}$. According to the above, it holds $r_{\hat n}=R^{n-1}(1)$ (or else $r_{n}=R^{{\hat n}-1}(1)$), where $R$ is the map defined in (\ref{zagier}).
Furthermore, an easy adaptation of (\cite{BI}, Theorem 2.3) shows that $R$ is topologically conjugated with the dyadic odometer (or von Neumann-Kakutani transformation \cite{VN}) $K:[0,1]\to [0,1]$, given by $K(1)\coloneqq 0$ and
$$
K(x)\coloneqq  x+ \frac 1 {2^{n-1}}+ \frac 1 {2^{n}} -1 \quad , \quad 1-\frac 1 {2^{n-1}} \leq x < 1-\frac 1 {2^{n}} \quad , \quad n\geq 1,
$$
via the map $\rho$ defined in (\ref{rho}), i.e.
\be\label{conju}
R = \rho^{-1} \circ K \circ \rho \, .
\ee
Finally, it is well known (see, e.g., \cite{KN}) that the map $K$ can be used to generate the Van der Corput sequence $\omega=(t_n)$, defined as follows: set first $t_1= 1/2$. Then, given $n\geq 2$, let $n=2^k+\sum_{l=1}^k s_l 2^{l-1}$ be its dyadic expansion and set $t_n=2^{-k-1}+\sum_{l=1}^k s_l 2^{-l}$. The first terms of $\omega$ are 
$$
t_1=\frac 1 2, \; t_2=\frac 1 4, \; t_3=\frac 34, \; t_4=\frac 1 8, \; t_5=\frac 5 8, \; t_6=\frac 3 8, \; t_7=\frac 7 8, \; t_8=\frac 1 {16}, \; \cdots
$$
Accordingly, we have $t_n=K^{n-1}(1/2)$, $n\geq 1$, and one readily gets the claim.
\end{proof}

\begin{remark} Note that the forward orbit of $1$ with $R$ is dense in $\R_+$, but it grows only logarithmically, as $R^{2^n-2}(1)=n$.
Moreover, according to \cite{CW} and \cite{Ne}, the following representation is in force: $R^{n}(1)=b(n)/b(n+1)$, $n\geq 0$, where $b(n)$ is the number of {\sl hyperbinary} representations of $n$, that is the
number of ways of writing the integer $n$ as a sum of powers of 2, each power being used at most twice. For instance $8=2^3=2^2+2^2=2^2+2+2=2^2+2+1+1$ and thus  $b(8)=4$.
\end{remark}

\noindent
The two maps $F$ and $R$ introduced above satisfy the following remarkable commutation rule:

\begin{proposition}\label{commu}  
For all $x \in \R_+$ we have $$R^m\circ F^n (x) = F^n \circ R^{2^nm} (x), \qquad n,m\geq 1$$
\end{proposition} 
\begin{proof} For the case $n=m=1$ the proof amounts to a straightforward verification, either by direct inspection or through the action of $F$ and $R$ on c.f.e.'s, that is (\ref{faction}) and
 \be
R: [a_0; a_1, a_2, \dots] \mapsto \left\{
  \begin{array}{ll} [1;a_1-1,a_2,\dots]\; , & \;   a_0=0\\ [0.3cm]
 [0; a_0, 1, a_1-1,a_2,\dots] \; , & \; a_0>0 \end{array} \right.
 \ee
 The general case easily follows by induction.
\end{proof}

\vskip 0.2cm
\noindent
Note that the map $R$ is invertible, with inverse 
\be R^{-1}(x)=1-\frac 1x+2\left[\frac 1x\right]
\ee
On the other hand, the map $F$ is two-to-one, with 
\be
F^{-1}(x)=\left\{\frac x{x+1},x+1\right\}
\ee
In particular, the set of $F$-pre-images of  $x=\frac pq$ coincides with the set of the descendants $\{\frac p {p+q}, \frac {p+q} q\}$ considered above (cf. Section \ref{dual}). 
\noindent
Therefore, as an ordered set, the tree $\hat {\cal T}$ can be generated both `horizontally', as the set of successive $R$-images of $1$, and `vertically', as the set of successive $F$-pre-images of $1$: ${\hat {\cal T}}= \cup_{n\geq 0}R^{n}(1) = \cup_{n\geq 0}F^{-n}(1)$, and, 
more specifically,
$$
\cup_{k= 0}^{2^n-2}R^{k}(1) = \cup_{k=0}^{n-1}F^{-k}(1) \quad , \quad \forall n\geq 1.
$$

\vskip 0.2cm
\noindent
Regarding the ergodic properties of these maps, we  start observing that $F$ possesses an absolutely continuous invariant measure $\nu$ , which can be computed explicitly: first the invariance means that $\nu=\nu F^{-1}$ where the latter is the measure which assigns to each measurable set $A\subset \R_+$ the number $\nu (F^{-1}(A))$.  Second, expressing this measure as  $\nu (dx) = h(x) dx$, the invariance property translates into the following functional equation for the density $h$:
$$
h(x)=\sum_{y\in  F^{-1}(x)} \frac {h(y)}{|F'(y)|}= \frac 1 {(1+x)^{2}} h\left( \frac x {1+x}\right)+h(x+1)
$$
and one immediately checks that a continuous solution is $h(x)=1/x$. Note that $h\notin L^1(\R_+, dx)$, that is $\nu$ is an infinite $F$-invariant a.c. measure. 
On the other hand, as the function $\rho$ establishes a topological conjugacy between $R$ and the dyadic odometer $K$ (see (\ref{conju})), it provides a topological conjugacy also between $F$ and the doubling map $D:[0,1]\to [0,1]$ (as shown in \cite{BI}), i.e.
\be\label{conju2}
F = \rho^{-1} \circ D \circ \rho \quad , \quad D(x)=2x \, ({\rm mod}\, 1)
\ee
The map $D$ acts as a shift on binary expansions and preserves the Lebesgue measure on the unit interval\footnote{This in particular entails that $F$ is {\sl chaotic} : is topologically transitive, its periodic orbits are dense and has sensitive dependence on initial conditions.}. 

\noindent
Since Lebesgue measure is preserved also by the invertible map $K$, the conjugacies (\ref{conju}) and (\ref{conju2}) ensure that both $F$ and $R$ leave invariant the probability measure $d\rho$.

\noindent
On the other hand, all orbits $\{ R^i(x) : i\geq 0\}$, $x\in {\overline \R}_+$ being dense, the dynamical system $({\overline \R}_+, R)$ is uniquely ergodic and therefore $d\rho$ is its unique invariant measure. In a different guise, the map $F$ possesses several invariant measures, two of which are $d\nu$ and $d\rho$, which are of course singular with respect to one another. In particular, as the entropy of the doubling map $D$ with respect to the Lebesgue measure is $\log 2$, this same value is also the entropy of $F$ with respect to the probability measure $d\rho$, which is therefore called the {\sl measure of maximal entropy} for $F$.

\subsection{An alternative ordering}

Proposition \ref{commu} can be viewed as expressing the fact that the "horizontal" action of the map $R$ respects the order induced by the "vertical" action of the map $F$ on the tree. Moreover, the conjugation (\ref{conju2}) between $F$  and $D$ can be obtained in two steps, passing via the map $\phi$ through the orientation preserving Farey map ${\tilde H}$, so that $F=\phi^{-1}\circ {\tilde H} \circ \phi$.
We can ask whether there is an orientation reversing version of the above constructions. For instance, if we consider the standard Farey map $H$, then the map $G=\phi^{-1}\circ H \circ \phi$, given by 
\be\label{effetilde}
{G} : x\mapsto  \left\{
  \begin{array}{ll} {\displaystyle \frac{x}{1-x}}\; , & \;   0 \leq x \leq 1\\ [0.3cm]
 \displaystyle {1\over x- 1} \; , & \; x>1 \end{array} \right.
 \ee 
is conjugated via $\rho$ with the tent interval map $T$, i.e. (\ref{conju2}) is replaced by $G = \rho^{-1} \circ T \circ \rho$.
Therefore, $d\rho$  is the measure of maximal entropy for $G$ as well. In addition, one easily verifies that $G$ preserves also the a.c.  measure with density $1/(x(1+x))$. We also note that $G(\Phi)=\Phi$ where $\Phi=(\sqrt{5}+1)/2$ is the golden mean. 
Since $|G'(\Phi)|=1+\Phi$ is a repelling fixed point. 

\noindent
Now, what is the map $S:{\overline \R}_+\to {\overline \R}_+$ which plays the role of $R$ in this orientation reversing setting? A close inspection based on continued fraction expansions leads to the following expression:  
\begin{align*}
{S} : x= & [a_0;a_1, a_2, \dots]\adjustbox{trim = 0mm 0mm 2mm -1mm, clip}{$\longmapsto$}\hspace{-2mm}\adjustbox{trim = -0.5mm 0mm 2mm -1mm, clip}{$\dasharrow$}\hspace{-2.71mm}\adjustbox{trim = -0.5mm 0mm 2mm -1mm, clip}{$\dasharrow$}\\ 
%{\huge \rotatebox[origin=t]{180}{$\Lsh$}}
\adjustbox{trim = 0mm 0mm 2mm -1mm, clip}{$\dasharrow$}\hspace{-1.27mm}\longrightarrow&\left\{
  \begin{array}{ll} { \displaystyle [0;n+1, a_{n}-1, a_{n+1}, \dots]  } \; , & \;  a_0=a_1=\cdots =a_{n-1}=1, \; a_{n}>1 \\ [0.3cm]
{  \displaystyle [a_1;a_2 , a_3, \dots]  } \; , & \;  a_0=0 \\ [0.3cm]
{ \displaystyle [0; \ell+2 ]} \; , & \; x= [{\underbrace 
{1;1,\dots , 1}_{\ell-1}},2] \end{array} \right.
\end{align*}
We also set $S  (0)=\infty$, $S(\infty)=1$ and $S(\Phi)=0$.
Now note that
$$
[{\underbrace {1;1,\dots , 1}_{\ell-1}},2] = \frac{F_{\ell+2}}{F_{\ell+1}}
$$
where $F_\ell$ be the $\ell$-th Fibonacci number, given by 
$$F_{-1}=1 \; , \;  F_0=0 \quad \hbox{and} \quad F_\ell=F_{\ell-1}+F_{\ell-2}  \quad , \quad \ell\geq 1$$
We then construct the sequence $(x_k)_{k \geq 0}$ as $x_k \coloneqq  F_k/F_{k-1}$, whose first elements are
$$
x_0=0 \; , \;  x_1=\infty \; , \; x_2=1 \; , \; x_3=2 \; , \; x_4=\frac 32  \; , \; x_5=\frac 53 \; ,  \quad \cdots$$
and observe that $S$ is continuous everywhere but at the points $x_k$, $k\geq 1$, where it is right-continuous.
An alternative expression for $S$ is thus the following:
\be\label{fibo}
{S} : x\mapsto   \frac{ F_{k} x -F_{k+1} } { (kF_{k}-F_{k-1}) x-kF_{k+1}+F_{k} } \quad , \quad x\in C_k
 \ee 
where
\be
C_{2r}=[x_{2r}, x_{2r+2}) \quad  , \quad 
C_{2r+1}=[x_{2r+3}, x_{2r+1}) \quad  , \quad r\geq 0
\ee
One checks that for all $x \in \R_+$ it holds
\be
S^m\circ G^n (x) = G^n \circ S^{2^nm} (x), \quad n,m \geq 1.
\ee

%{ \displaystyle \frac{1}{k+2} } \; , & \; x= x_k \end{array} \right. 

\section{Motions on the modular surface}

$F$ can be obtained as the {\sl factor map} of a first return map for the geodesic flow on the modular surface. Let us briefly recall what does this mean. 

\noindent
Let $\HP= \set{z=x+iy\ :\ x \in \R,\ y\in \R_+}$ be the upper half-plane, viewed as a Riemmanian manifold with hyperbolic metric
$ds^2 = (dx^2 + dy^2)/y^2$. Set moreover $M= \Gamma \setminus \HP= \{ \Gamma z : z\in \HP\}$, with $\Gamma=PSL(2,\Z)$, endowed with the quotient topology.
We recall that the Fuchsian group $\Gamma$ has two generators $U$ and $V$, which can be chosen as
$U=\mat{0&1\\-1&0}$ and $V=UB^{-1}=AU=\mat{0&1\\-1&1}$.
It holds moreover $U^2=V^3=I$ (so that $\Gamma$  is not a free group). 

\vskip 0.2cm
\noindent
Let  $\varphi_t : SM\to SM$ be the geodesic flow on  the unit tangent bundle of $M$, and
let us construct a subset of $SM$ which is met infinitely many times by each $\varphi_t$-orbit. To this end set
$$
{\cal I} =\set{z=x+iy \ :\ x=0,\ y\in \R^{+}} \subset \HP
$$ 
and consider the section $C$ made by the projections on $SM$ of all vectors of $S\HP$ having base point on ${\cal I}$ and right-oriented, that is vectors of the form $v=(z,\theta )$ with $z\in {\cal I}$ and $\theta \in (\pi ,2\pi )$. 
One easily sees that the elements thus selected are all distinct. There are however $\varphi_t$-orbits which do not visit $C$ infinitely often. These are exactly the projections of geodesics which either start or end in a cusp of $PSL(2,\Z)$, that is a rational point on the real line. On $SM$ these orbits converge towards  (or come from) the cusp at infinity and for this reason they are called {\sl scattering geodesics}. They form of course a set of zero measure.  

\vskip 0.2cm
\noindent
Now, a vector $v\in S\HP$ whose projection lies in $C$ can be described by the two asymptotic coordinates  $u$ and $w$ which identify  the geodesic $\gamma (v,t)$ having tangent vector $v$ at $t=0$. Hence,
$$
C\coloneqq \set{(u,w) \ :\  u<0<  w } 
$$
In turn $C$ can be decomposed as
$C=C_1\cup C_2$ where
$$
C_1=\{(u,w)\, : \, u< 0< w <1\}\quad , \quad C_2=\{(u,w)\, : \, u< 0, \;  w >1\}
$$
The next figure\errata{Needs checking after formatting is complete} shows a geodesic $\gamma$ such that the projection on $SM$ of $\gamma \cap {\cal I}$ belongs to $C_2$.
%\captionsetup[figure]{labelsep=space}
\begin{figure}[h!]
\begin{center}
%\hspace{-0.5cm}
%\fbox{\includegraphics[width=14cm, trim = 0.1cm 0.1cm 0.41cm 1cm]{(u,w)}}
\includegraphics[width=12cm, trim = 0cm 0.1cm 0.35cm 1cm, clip]{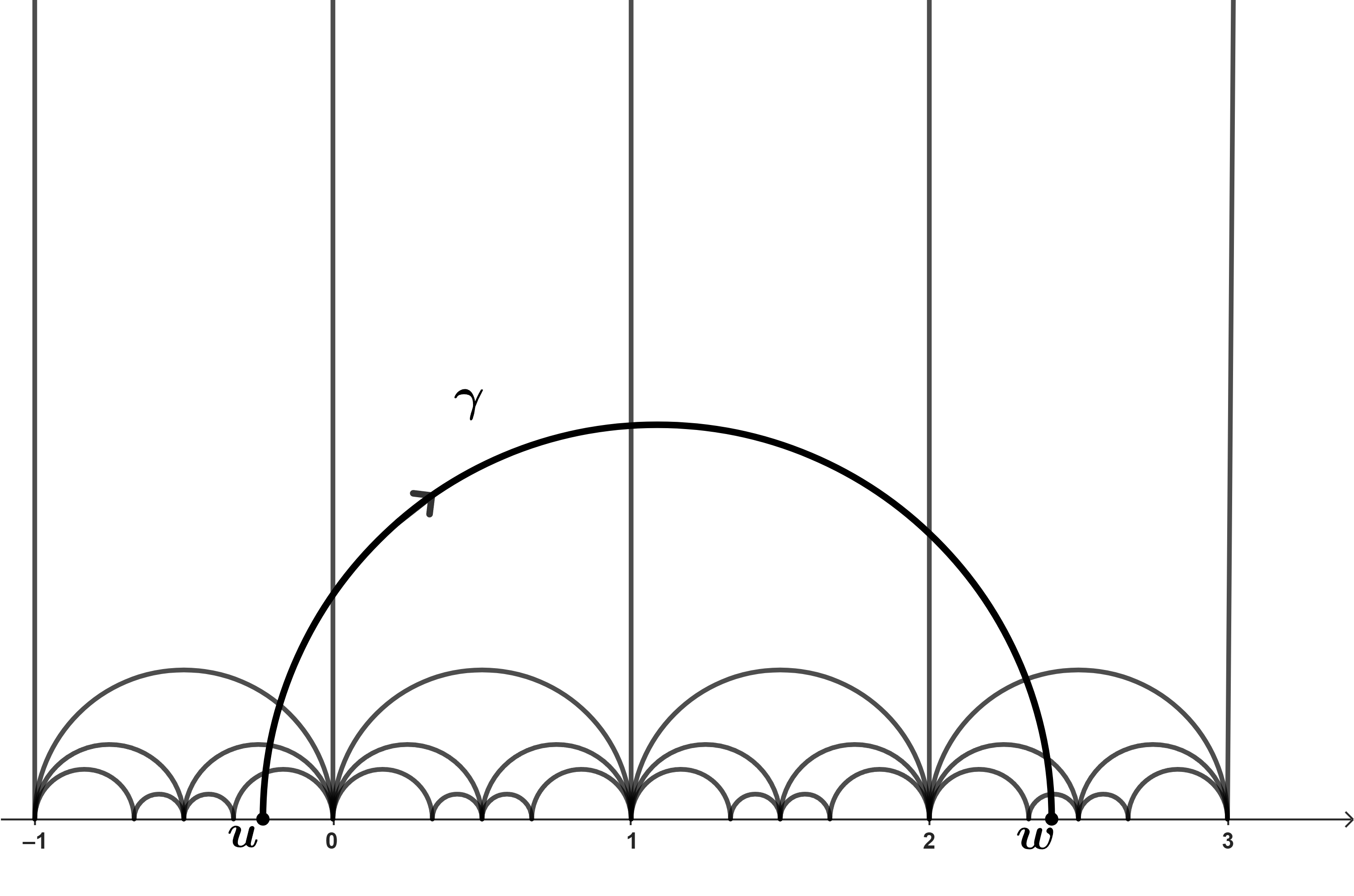}
\end{center}
\vspace{-5mm}
\caption{}
\label{fig:(u,w)}
\end{figure}
%\captionsetup[figure]{labelsep=colon}

\noindent
We now construct the {\sl first return map} $T_C:C\to C$ which sends each intersection of a $\varphi_t$-orbit with $C$ to the next one. To this end, we consider the geodesic triangle 
$\bG$ with vertices  $0$, $1$ and $\infty$, that is
$$
\bG =\{z\in \HP \,| \, 0<\Re\, z <1, |z-\frac{1}{2}|>\frac{1}{2}\}
$$
Its three sides are equivalent w.r.t. $PSL(2,\Z)$:
 $\hat {01}$ and $\hat {1\infty}$ are mapped to ${\cal I}$ by the transformations $UV^2 \equiv A^{-1}: z\to z/(1-z)$ and $UV \equiv B^{-1}: z\mapsto z-1$ respectively. Now, suppose that 
the projection of $v\in S\HP$ lies in $C$ and has coordinates $(u,w)$. There are two possibilities: if the projection of $v$ lies in $C_2$ (so that the geodesic $\gamma$ determined by $v$ leaves ${\bG}$ through $\hat {1\infty}$), then it is mapped by $B^{-1}$ to $(u-1,w-1)$; if instead the projection of $v$ lies in $C_1$ (so that $\gamma$ leaves ${\bG}$ through $\hat {01}$), then it gets mapped by $A^{-1}$ to $(\frac u {1-u}, \frac w {1-w})$.
Therefore the first return map on $C=C_1\cup C_2$ is 
\be
T_C: (u,w) \mapsto \left\{
  \begin{array}{ll}
  \left(\displaystyle \frac{u}{1-u}, \frac {w}{1-w} \right) \; , &   (u,w) \in C_1 \\ [0.3cm]
\; \,(\, u-1, w-1\,) \; \;\;, &   (u,w) \in C_2 \end{array} \right.
\ee
The action of $T_C$ on the second coordinate finally yields the {\sl factor map}  $F:\R_+\to \R_+$ given by (\ref{effe}). 
\begin{figure}[h!]
\begin{center}
%\fbox{\includegraphics[width=\textwidth]{scattering1}}
\includegraphics[width=12cm, trim = 0cm 0.1cm 0cm 0.2cm, clip]{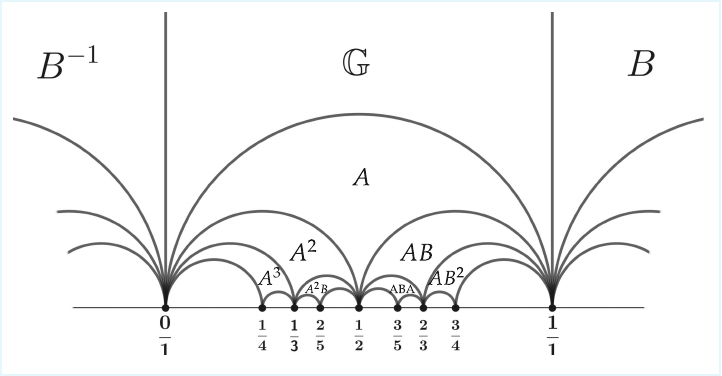}
\caption{}
\label{fig:TessellationG}
\end{center}
\end{figure}
%\begin{center}
%\includegraphics[width=14.0cm]{ft1.pdf}
%\end{center}
\vspace{-8mm} \\

\noindent
Now, referring to the figure above,\errata{Needs checking after formatting is complete} one can produce a tessellation of $\HP$  by taking all the images of the geodesic triangle $\bG$ with the isometries $A$ and $B$ (acting as M\"obius transformations). Moreover, a direct consequence of the generating rule (\ref{F}) is that, given $x=p/q$, the matrix product $X$ dealt with in Proposition \ref{codaggio}, as well as the corresponding binary sequence $\sigma (x)\in \{0,1\}^*$, are in a one-to-one correspondence with the coding w.r.t. the above tessellation of the scattering geodesic $c_{p/q}$ which converges to $p/q$, the central cusp of  the geodesic triangle $X(\bG)$ (see \cite{Kn}). \\
\begin{figure}[h!]
\begin{center}
%\fbox{\includegraphics[width=\textwidth]{scattering2}}
\includegraphics[width=12cm, trim = 0cm 0.1cm 0cm 0.2cm, clip]{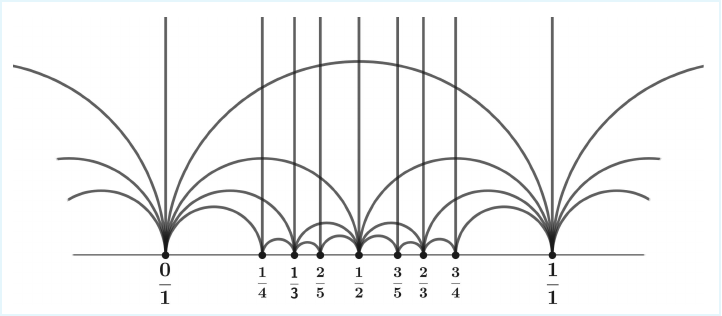}
\caption{}
\label{fig:scattering}
\end{center}
\end{figure}
%\begin{center}
%\includegraphics[width=10.0cm]{scattering}
%\end{center}
\vspace{-8mm} \\

\noindent
In a similar fashion as finite paths on $\cal T$ correspond to scattering geodesics on $\HP$, we can establish a correspondence between FC words and {\sl Ford circles}. These are a countable family of circles orthogonal to the sides of the just mentioned geodesic triangles. Each of them, denoted $C_{\frac p q}$, is tangent to $\R$ in some rational point $p/q$, and has diameter $1/{q^2}$. The largest circles have thus unit diameter and correspond to $C_n$, $n\in \Z$ (the following\errata{Needs checking after formatting is complete} picture shows $C_0$, $C_{\frac 13}$, $C_{\frac 12}$, $C_{\frac 23}$ and $C_1$).\\
\begin{figure}[h!]
\begin{center}
\includegraphics[width=11cm, trim = 0.1cm 0.2cm 0.15cm 0.1cm, clip]{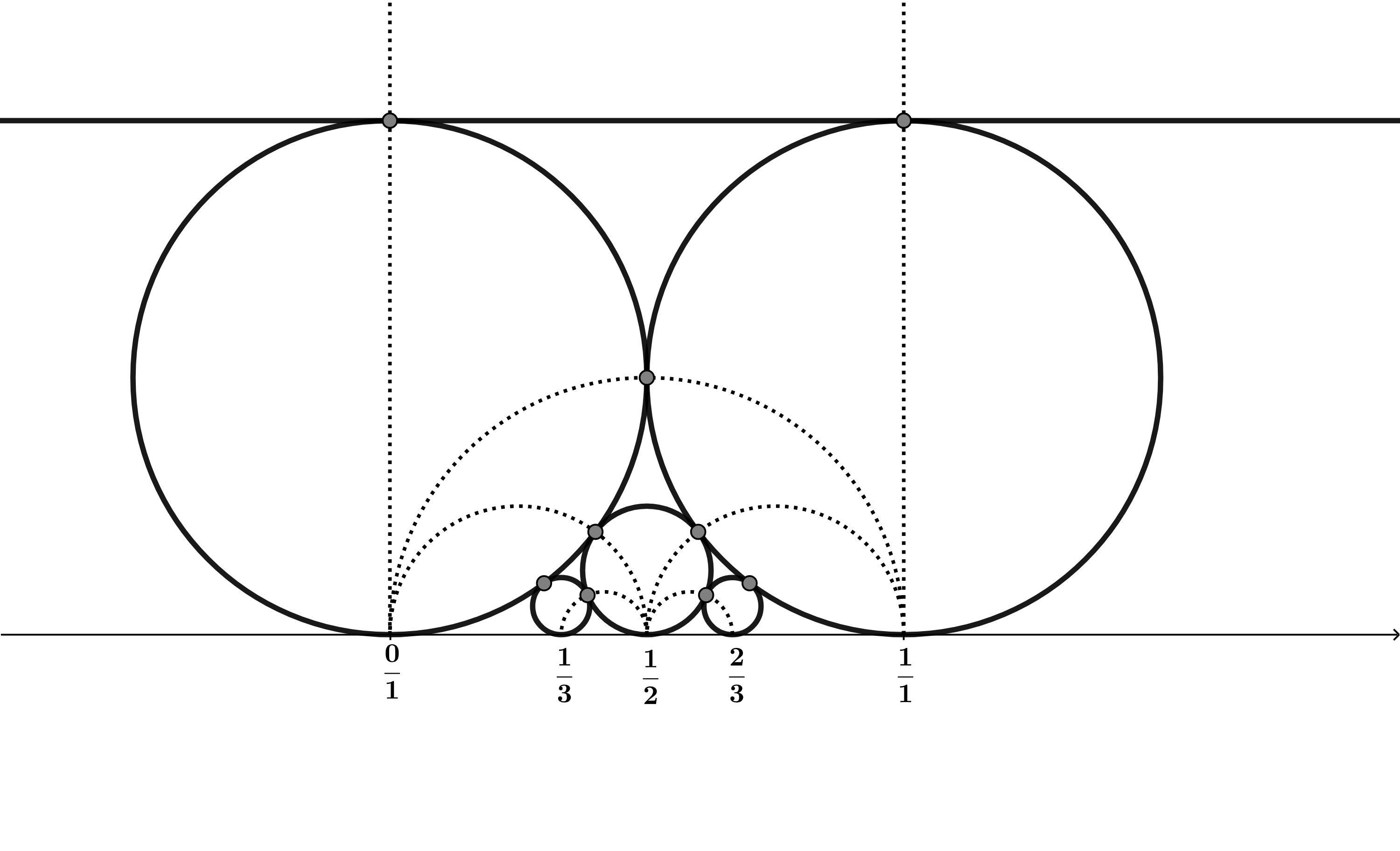}
\caption{}
\label{fig:Fordcircles}
\end{center}
\end{figure}
%\begin{center}
%\includegraphics[width=10.0cm]{ford}
%\end{center}

\newpage
\noindent
Clearly, each Ford circle $C_{\frac p q}$ with $\frac pq \geq 0$ corresponds to a unique FC word $w$ with $\frac pq = \frac {|w|_1}{|w|_0}$, and vice versa.
\vskip 0.1cm 
\noindent
Ford circles and scattering geodesics are related as follows:
\noindent
first, the image with $X_{\frac pq} =  \left(
  \begin{array}{cc}
  n & m\\
  t & s
  \end{array}\right) \in SL(2,\Z)$ of the {vertical geodesic} $I=\{z=ie^{\tau} :  \tau \in \R\}$ 
is a geodesic connecting $X_{\frac pq}(0)= \frac ms$ and $X_{\frac pq}(\infty)= \frac nt$.  
$X_{\frac pq}(\bG)$ is a Farey triangle with central cusp in $\frac pq=\frac{m+n}{s+t}$.
\noindent
 If, instead, we apply $X_{\frac pq}$ to the {positive and negative horocycles} of $v=(i,0)\in T\HP$, namely
the horizontal line $H^+=\{z=i+ \tau : \tau\in \R\}$ ($B$-invariant) and  the circle $H^-=\{z=\frac i{1+i\tau} : \tau \in \R \}$ ($A$-invariant) 
 we obtain two Ford circles: 
\begin{itemize}
\item $C_{\frac nt}$, of diameter $ \frac1{t^2}$ and tangent to $\R$ in $ \frac nt$, 
 \item $C_{\frac ms}$, of diameter $ \frac1{s^2}$ and tangent to $\R$ in $ \frac ms$,
\end{itemize}
 which touch each other at the point $X_{\frac pq}(i)$.
The ``child" circle $C_{\frac pq}$ touches the cusp at $\frac pq$,
 and the ``parents" circles $C_{\frac nt}$ and $C_{\frac ms}$ at $X_{\frac pq}B(i)$ and $X_{\frac pq}A(i)$, respectively. Finally, 
the geodesics that cross $C_{\frac pq}$ perpendicularly (in particular  $c_{{\frac pq}}$) converge at the cusp.

\vskip 0.1cm 
\noindent
{\sc Example.} $X_{\frac 12}=A= \left(
  \begin{array}{cc}
  1 & 0\\
  1& 1
  \end{array}\right)$,  $C_{\frac 12}= A^2(H^+)=AB(H^-)$ (see the figure above).\errata{Needs checking after formatting is complete}
\vskip 0.1cm 
\noindent
One easily checks that two Ford circles $C_{\frac p q}$ e $C_{\frac {p'} {q'}}$, with $\frac p q < \frac {p'} {q'}$, are either tangent to each other or they do not intersect, and the former situation occurs whenever $p'q-pq'=1$. Moreover, three Ford  circles $C_{\frac p q}$, $C_{\frac {p'} {q'}}$ and $C_{\frac {p''} {q''}}$ with $\frac p q < \frac {p''}{q''}<\frac {p'} {q'}$ are tangent to each other if and only if $\frac {p''}{q''}=\frac {p}{q}\oplus \frac {p'}{q'}$ (see, e.g., Theorems 5.6 and 5.7 in \cite{Ap}). \vspace{1mm} \\
We can say more, but first we briefly present the classical correspondence between a matrix $X \in PSL(2,\R)$ and $v=(z,\theta) \in S\HP$. Given $ v=(z,\zeta) \in S\HP$, with $z \in \HP$ and $\zeta \in T_z\HP \simeq \C$, we can identify $S\HP$ with $PSL(2,\R)$ by corresponding $v$ to the unique element $ g \in PSL(2,\R)$ such that $z=g(i)$ and $\zeta=\dd{g(\zeta_0)}=g'(z)\zeta_0$, where $\zeta_0$ is the unit vector tangent to the imaginary axis. One can also write the unit tangent vector as $\zeta = \operatorname{Im}(z)e^{i(\theta+\frac{\pi}{2})}$ where $\theta$ is the angle formed by $\zeta$ with the vertical line, measured counterclockwise. By identifying $\zeta$ with $\theta$, we obtain the parametrization $v=(z, \theta)$ for the points in $S\HP$, and
$$
(z, \theta) = \left( g(i) , \beta_g(0) \right)
$$
where $g= \begin{pmatrix} a &b \\ c &d \end{pmatrix}$ is given by
\be
\label{eq:PSL-H.2}
z = g(i) = \frac{b+ia}{d+ic} \, , \quad \theta = \beta_g(0) = -2 \arg(d+ic) = -2 \tan^{-1} \left(\frac{c}{d}\right) 
\ee
In this way, the action of the \emph{positive} and \emph{negative horocyclic flow} $h^+_t$ and $h^-_t$ on $PSL(2,\R)$ corresponds to the right multiplication by one-parameter subgroups of matrices
\be
\label{horocyclicflows}
n^+_t = \begin{pmatrix} 1 & t \\ 0 & 1 \end{pmatrix}, \, \, h^+_t \longleftrightarrow g\,n^+_t \quad \mbox{and} \quad n^-_t = \begin{pmatrix} 1 & 0 \\ t & 1 \end{pmatrix}, \, \,  h^-_t \longleftrightarrow g\,n^-_t
\ee
This also assures us of the commutativity between isometries and flows, since the former act from the left while the latter act from the right.
Finally we can say the following: consider the correspondence between an element $x \in \mathcal{T}$ and $X \in SL(2,\Z)$, given by \eqref{ide}, and the correspondence between a matrix $X \in SL(2,\Z)$, viewed as an element of $PSL(2,\R)$, and $v=(z,\theta) \in S\HP$, given by \eqref{eq:PSL-H.2}. This gives a correspondence between elements in $\mathcal{T}$ and points $z \in \HP$, as follows:
\be
\label{eq:correspondence}
x = \frac{m}{s} \oplus \frac{n}{t} \longrightarrow X  = \begin{pmatrix} n&m\\t&s \end{pmatrix} \longrightarrow v = \left( X(i) , \beta_X (i) \right)  \longrightarrow X(i)
\ee
recalling that $\beta_X (i)= -2\tan^{-1}(t/s)$. \vspace{0.5mm} \\
However, this correspondence is not a bijection since the same point in $\HP$ can be associated to multiple point in $S\HP$ and hence to multiple $X \in SL(2,\Z)$ which are not even associated to some $x \in \cal T$. But considering the direction from $x \in \mathcal{T}$ to $z \in \HP$, which is well defined, we get a correspondence between $x$ and $z=X(i)$.\\
Moreover, for our scope, we just need to prove that:
$$
 X_1 = \begin{pmatrix} n&m\\t&s \end{pmatrix} \quad \mbox{ and } \quad X_2 = \begin{pmatrix} m&-n\\s&-t \end{pmatrix}
$$ 
correspond to $v_1, v_2 \in S\HP$ with $z_1 = z_2 $ and opposite vectors $\theta_1$ and $\theta_2$. \\
But this is easily shown considering:
$$
\frac{-n +mi}{-t + si} = \frac{-n +mi}{-t + si} \cdot \frac{-i}{-i} = \frac{m +ni}{s + ti}
$$
and, recalling that $\tan^{-1}(x) + \tan^{-1}(\frac{1}{x}) = \pm \frac{\pi}{2}$, 
$$
 -2\tan^{-1}\left(\frac{t}{s}\right) + 2\tan^{-1}\left(\frac{s}{-t}\right) = -2\left(\tan^{-1} \left(\frac{t}{s}\right) + \tan^{-1} \left(\frac{s}{t}\right) \right) = \pm \pi .
$$
So, we have a direct way to determine both $x$ and $z$ from $X \in PSL(2,\Z)$, where $z$ is obtained in the canonical way, and 
\be
\label{eq:correspondence2}
x =  \frac{m}{s} \oplus \frac{n}{t} =  \frac{n}{t} \oplus \frac{m}{s} \eqcolon \frac{-n}{-t} \oplus \frac{m}{s}
\ee
\vskip 0.5cm
\noindent
{\sc Example.} As in the previous example, we have $C_\frac{1}{2} = A^2(H^+) = AB(H^-)$, which indeed is the negative horocycle for $v_1=(z_1,\theta_1)$, with $z_1 \leftrightarrow  A^2 \leftrightarrow \frac{1}{3}$ and the positive horocycle for $v_2=(z_2,\theta_2)$, with $z_2 \leftrightarrow AB \leftrightarrow \frac{2}{3}$ (see Figure \ref{fig:Fordcircles}).
\vskip 0.5cm
\noindent
With the elements presented thus far, we can show that the horizontal movement on $\cal T$ corresponds to horocyclic flows along Ford circles. To this end we present first the following.
\begin{lemma}\label{ford}The horocyclic flow with unit time on a Ford circle moves from a tangency point with another Ford circle to the next one.
\end{lemma}
\begin{proof}
From the content of this section, we know that the Ford circles associated with $\frac{1}{0}$ (the horizontal line) and $\frac{0}{1}$ can be mapped to any other Ford circle $C_x$ via an isometry. We can consider the Ford circle $C_x$ associated with $\frac{p}{q}$ and the tangency point with another Ford circle $C_{x'}$ associated with $\frac{p'}{q'}$. Then, both horocyclic flows, with either negative or positive unit time, are mapped to the respective flows on the Ford circles $C_\frac{1}{0}$ and $C_\frac{0}{1}$. For these, it can be directly checked that, moving with unit time (positive or negative), we are moving from the starting tangency point $z = i$ to the next one in the corresponding direction along the corresponding horocycle. This proves the lemma. 
\end{proof} 

\vskip 0.1cm 
\noindent
To state the next result, for any positive integer $t$ we set:
$$
 A^t \coloneqq  \left(
  \begin{array}{cc}
  1 & 0\\
  t& 1
  \end{array}\right) \equiv h^-_{t}  \, , \qquad
  D^t\coloneqq  B^{-t}= \left(
  \begin{array}{cc}
  1 & -t\\
  0& 1
  \end{array}\right) \equiv h^+_{t} 
 $$
so that, in particular, $A^1= h_{t=1}^- =A $ and $D\equiv D^1= h_{t=-1}^+=B^{-1}$.\\
Then, the horocyclic flows with time $t$ correspond to either $A^t$ or $B^t$, as in \eqref{horocyclicflows}. Moreover, as shown in \eqref{eq:correspondence} and \eqref{eq:correspondence2}, we recall that each fraction $x$ in $\cal T$ (and $\hat {\cal T}$) corresponds to the tangency point between the parents of the Ford circle $C_x$, and vice versa. \vspace{2mm} \\
We can now state the following:
\begin{theorem}
\label{th:cyclic}
The horizontal displacement on $\cal T$, starting at the root 1 and moving from left to right on each level, corresponds to clockwise motion along Ford circles. More precisely, assume that we reached $x = r_m$, the $m$-th element of $\cal T$, as in Theorem \ref{regola}, with $depth(x) = n$. Then, the move to the next element $y = r_{m+1}$ corresponds to the following displacement (via horocyclic flow) on Ford circles: 
\begin{itemize}

\item if $x$ is the rightmost element in a level, i.e. $m=2^n-1$, then moving to $y$ corresponds to applying $D^{n-1}A^n$ for $n$ even, and $A^{n-1}D^n$ for $n$ odd;

\item if, instead, $x$ is either the leftmost or an inner element in a level, i.e. $m=2^{n-1}+(k-1)$ for some $1\leq k < 2^{n-1}$ and $k=2^{p-1} ({\rm mod} \, 2^p)$, with $1 \leq p \leq n-2$, then moving to $y$ corresponds to applying
$A^{1+2(p-1)}$ if $n=k ({\rm mod} \, 2)$, $D^{1+2(p-1)}$ otherwise.
\end{itemize}
\end{theorem}
\begin{proof}
Firstly, it is important to note that when considering the horocyclic flows, each time we move from one Ford circle to another tangent to it, the vector switches direction from inward to outward, or vice versa. This means that, since the movement is clockwise, we transition from the positive horocyclic flow with negative time $h^+_{-t} \equiv D^t$ (to the left of the vector) to the negative horocyclic flow with positive time $h^-_t \equiv A^{t}$ (to the right of the vector), or vice versa, from $h^-_t \equiv A^{t}$ to $h^+_{-t} \equiv D^t$.
Since each level $n>1$ of the tree contains an even number of elements, as we move along the level, we perform an odd number of swaps between horocycles before reaching the last element $\frac{n}{1} \in \mathcal{T}$. This element corresponds to $z = (n-1) + i \in \HP$, i.e. the point of tangency between $C_\frac{1}{0}$ and $C_\frac{n-1}{1}$ (the parents of $C_\frac{n}{1}$). As a result, the vector $v_\frac{n}{1}$ will point in the opposite direction compared to $v_\frac{n-1}{1}$ w.r.t. $C_\frac{1}{0}$. Therefore, when moving from one level to the next, say from $n$ to $n+1$, we alternate between $D^{n-1}A^{n}$, when $n$ is odd, and $A^{n-1}D^{-n}$, when $n$ is even. In this way, the direction of the vector $v$ is reversed two more times, and the next level $n+1$ start from $\frac{1}{n+1}$ with the vector in the opposite direction compared to $\frac{1}{n}$. Thus, the horocyclic flow that begins at the start of a level $n$ of the tree correspond to $A$ if $n$ is odd, and to $D$ if $n$ is even.\\
Now let $x=r_m$, where $m=2^{n-1} + (k-1)$,with $ 1 \leq k \leq 2^{n-1}$, so that it is the $k$-th element of the $n$-th level of $\mathcal{T}$. If we want to move horizontally to the next element $r_{m+1}$, we have two possibilities: either $k<2^{n-1}$, in which case we move to position $k+1$ on the same level, or $r_{m+1}$ is the first element of the next level $n+1$. However, we have already discussed this case, so, from now on, we will consider $k < 2^{n-1}$.\\
If $k$ is odd, then $x$ is the left child of its parent node $x'$, and $r_{m+1}$ is the right child. In $\HP$, each of these two corresponds to the tangency points between the Ford circle $C_{x'}$ of $x'$ and the Ford circle of the other parent. Therefore, as in Lemma \ref{ford}, moving from one point to the next along $C_{x'}$ corresponds to the horocyclic flow with $|t|= 1$, which, depending on the orientation of the vector $v$, corresponds to $A$ if $n$ is odd, or $D$ if $n$ is even.\\
If, instead, $k$ is even, then we have a right child, and its parent is different from the parent of $r_{m+1}$. Indeed, we need to go back at least two levels to find a common ancestor. Considering the structure of the tree, one can see that for $k=1,2,3,\ldots,2^{n-2}, \ldots, 2^{n-1}$, the number of steps needed to reach the common ancestor is $1$, $2$, $1$, $3$, $1$, $2$, $1$, $4, \ldots, 1$, $n-1$, $1$, \ldots, $1$. In general, for $k = 2^{p-1} \pmod{2^{p}}$, for $1\leq p \leq n-2$ we need $p$ steps. This can be easily proven by induction on the level of the tree. For $n = 2$, it is trivially true. Assuming the formula holds for levels up to $n$, it follows that, by construction, for all the new left children, which correspond to $k=1 \pmod{2}=2^0 \pmod{2^1}$, the formula holds. For a given right child $x$, the common ancestor with the node directly to its right, which coincides with the common ancestor of its parent $x'$ with the node to its right, is one step further than the number of steps required from its parent $x'$. By induction, from $x'$, corresponding to $k' = 2^{p-1} \pmod{2^{p}}$, we need $p$ steps, so from $x$ we will need $p+1$. From one level to the next the nodes duplicate, and $x$ will be at the position $k = 2k'$ so that $k = 2^{p} \pmod{2^{p+1}}$, as required.\\
We have that both $r_m$ and $r_{m+1}$ correspond to points on the Ford circle associated with the (nearest) common ancestor $y \in \mathcal{T}$, specifically to the points of tangency with their respective parent. On the horocycle, between them, there are $2(p-1)$ points, where $p$ is the number of steps required to reach the common ancestor. Indeed, all the nodes traversed while moving up from $r_m$ to the ancestor form a Farey pair with $y$, as do the nodes traversed to reach down to $r_{m+1}$, and, by the properties of $\mathcal{T}$ and the Ford circles, these are all and only the points that lie between them. Thus, following the ideas in the proof of Lemma \ref{ford}, this movement corresponds to the horocyclic flow with time $|t|=1 + 2(p-1)$. The exact one, $A$ or $D$, depends on $m$, and, more directly, on $n$ and $k$. As we have seen, for $n$ even, odd $k$ corresponds to $D$ and even $k$ corresponds to $A$, while the reverse is true when $n$ is odd.
\end{proof}
\vskip 0.5cm 
\noindent
We already showed how the scattering goedesics in $\HP$ are correlated with the vertical movement on the Stern-Brocot tree $\mathcal{T}$. With this theorem, we established a parallel between Ford horocycles, which are orthogonal to the geodesics defined in the Farey tessellation, and the horizontal movement on $\mathcal{T}$.
\begin{remark}
The repeated horizontal movement on $\mathcal{T}$ can be interpreted geometrically as a cyclical movement along the upper arcs of the Ford circles and, dynamically, as a repeated composition of horocyclic flows. This corresponds to a repeated right multiplication of matrices, expressed as:
$$
\begin{aligned}
&(A) D \\
&(A D^2) A D^3 A \\
&(D^2 A^3) D A^3 D A^5 D A^3 D \\
&(A^3 D^4) A D^3 A D^5 A D^3 A D^7 A D^3 A D^5 A D^3 A \\
&(D^4 A^5) \ldots
\end{aligned}
$$
where the brackets correspond to the jump to the next level on $\mathcal{T}$, or equivalently, to the return to $i$ in $\HP$ and subsequent descent towards $X_\frac{1}{n+1} (i) \leftrightarrow \frac{1}{n+1}$.
\end{remark}
\begin{remark}
If one want to consider the horizontal movement on the $n$-th level of $\mathcal{T}$ as composition of horocyclic flows but always resetting and starting from $(i,0) \in S\HP$, we would have
\begin{gather*}
(I_2)
\\ (A) D
\\ (A^2) D A^3 D
\\ (A^3) D A^3 D A^5 D A^3 D;
 \\(A^4) D A^3 D A^5 D A^3 D A^7 D A^3 D A^5 D A^3 D
\\ \vdots
\end{gather*}
which more clearly show the palindromic and symmetric nature of the movement along a level of $\mathcal{T}$, obviously already present in Theorem \ref{th:cyclic}.
\end{remark}
To conclude, we provide figures to visualize the motions described in Theorem \ref{th:cyclic}. In the first figure, we indicate the direction of traversal of the circles, which will be omitted in the subsequent figures, as it remains the same, i.e., clockwise. Additionally, clockwise is considered the negative direction along the horizontal line $C_\frac{1}{0}$.  After the first two figures we will omit vectors and points to reduce clutter. Moreover, in all figures, we color-code the horocyclic flows: red (\textcolor{red}{$h^-_t$}) for the negative horocycle $H^-$, associated with positive time , and blue (\textcolor{blue}{$h^+_{-t}$}) for the positive horocycle $H^+$, associated with negative time\footnote{Cf. the correspondence \eqref{horocyclicflows}.}. Specifically, red represents $A^t$, and blue represents $D^t$, where $t-1$ denotes the number of tangent points that must be surpassed to reach the end of the arc.
A note is due: in the figures showing the movement on the $n$-th level, we have added, for completeness, the descent from $\frac{1}{1}$ to the first element of the $n$-th level, which would not be included in the movement through the level. Visually, it correspond to the leftmost colored arc, descending from $i$ along $C_\frac{0}{1}$. \newpage
\begin{figure}[h]
 \centering
   \begin{minipage}{0.50\textwidth}
    \centering
\includegraphics[width=6.5cm, trim =2.6mm 0.5mm 4mm 1.5mm, clip]{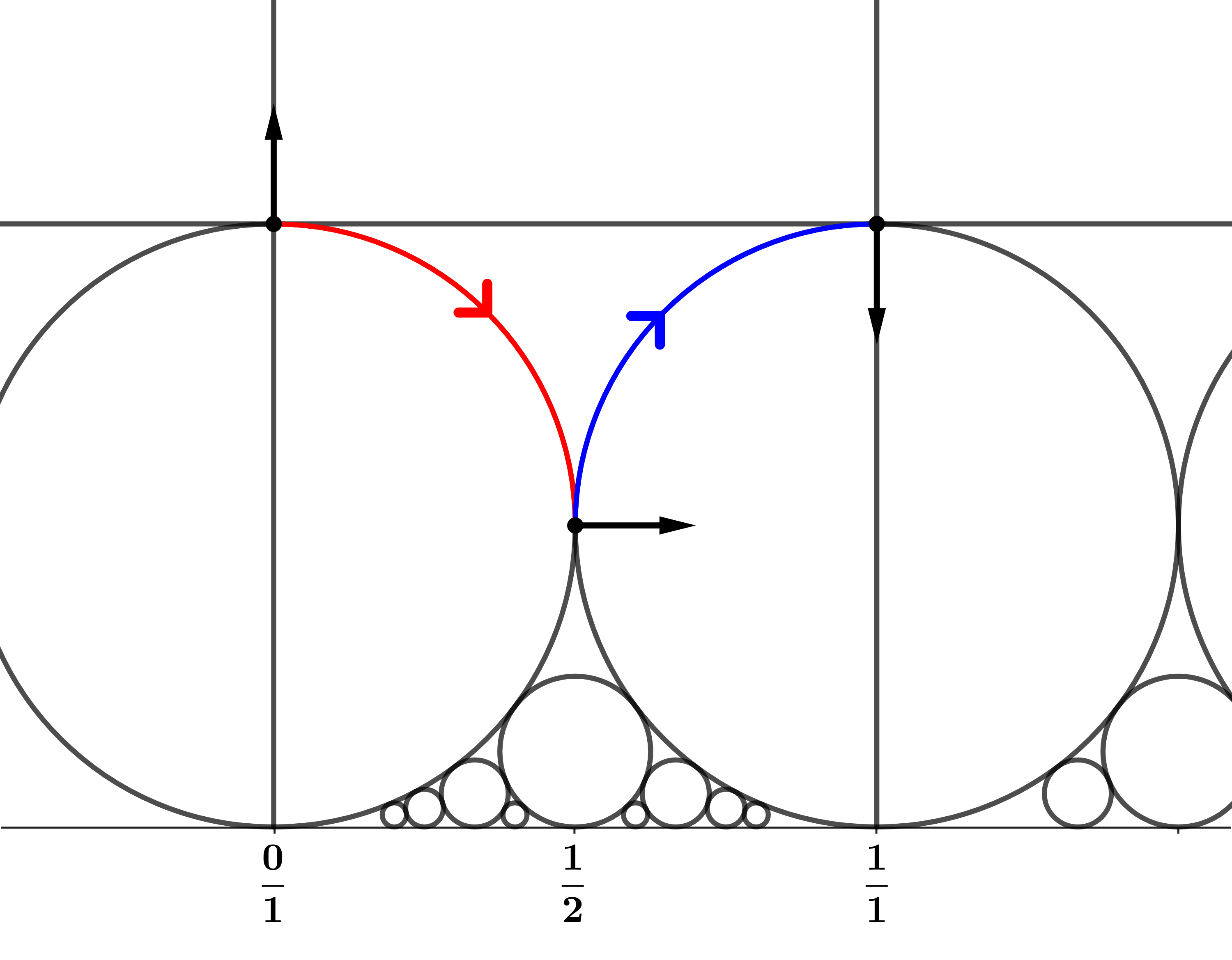}
    \end{minipage}\hfill
    \begin{minipage}{0.50\textwidth}
     \centering
\includegraphics[width=6.5cm,  trim = 2.6mm 0.5mm 4mm 1.5mm, clip]{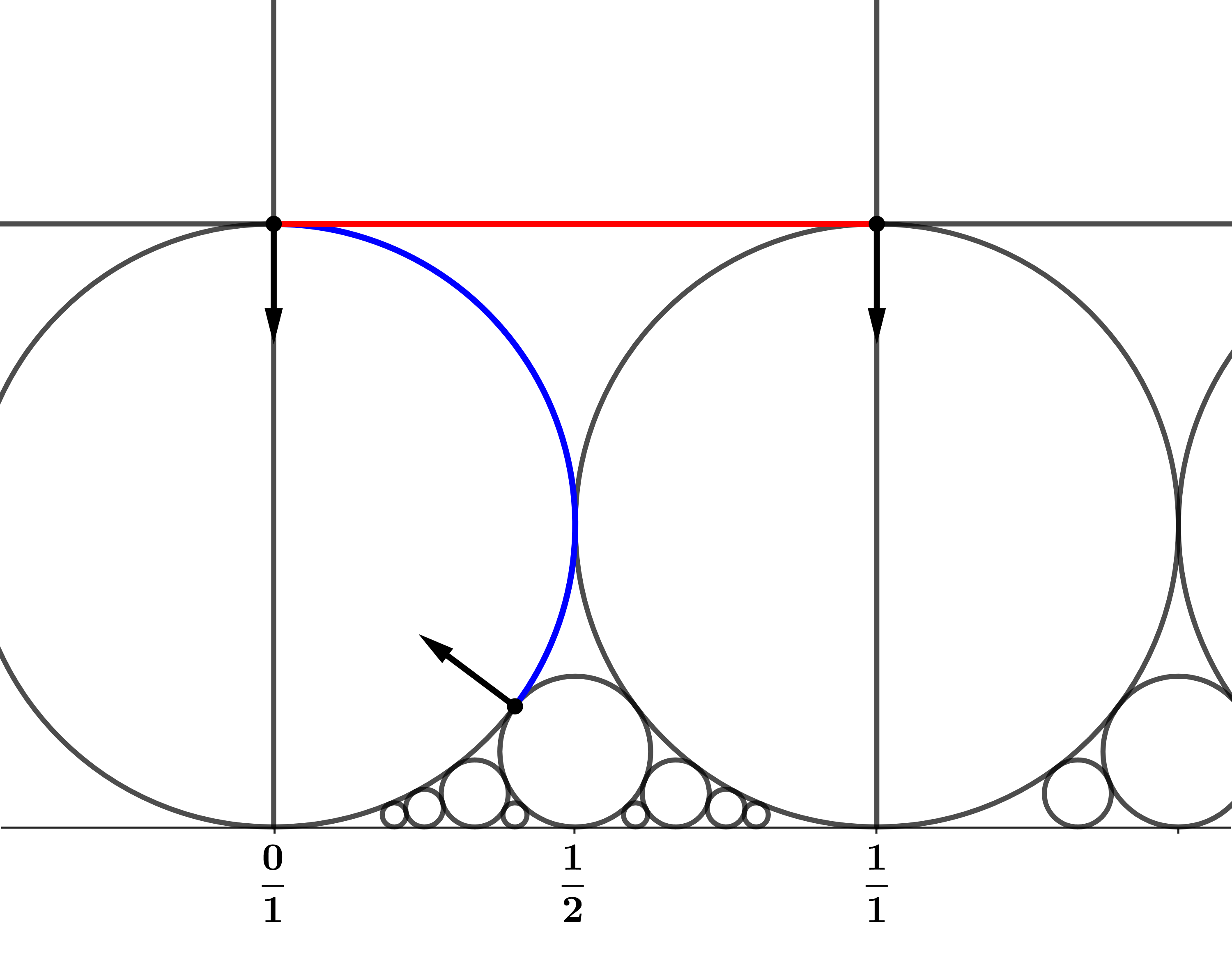}
   \end{minipage}
\label{fig:Level1}
\caption{On the left, movement on the second level of $\mathcal{T}$ with [\textcolor{red}{$h^-_1$}] (the descent) followed by \textcolor{blue}{$h^+_{-1}$}. On the right, transition to the third level with \textcolor{red}{$h^-_1$} followed by \textcolor{blue}{$h^+_{-2}$}}
%\caption{On the left, movement on the second level of $\mathcal{T}$ with $n^-_1$ (red) followed by $n^+_{-1}$ (blue). On the right, transition to the third level with $n^-_1$ (red) followed by $n^+_{-2}$ (blue) }
\end{figure}
\begin{figure}[h!]
\begin{center}
\includegraphics[width=12.5cm, trim = 0.1mm 2mm 0.4mm 3.5mm, clip]{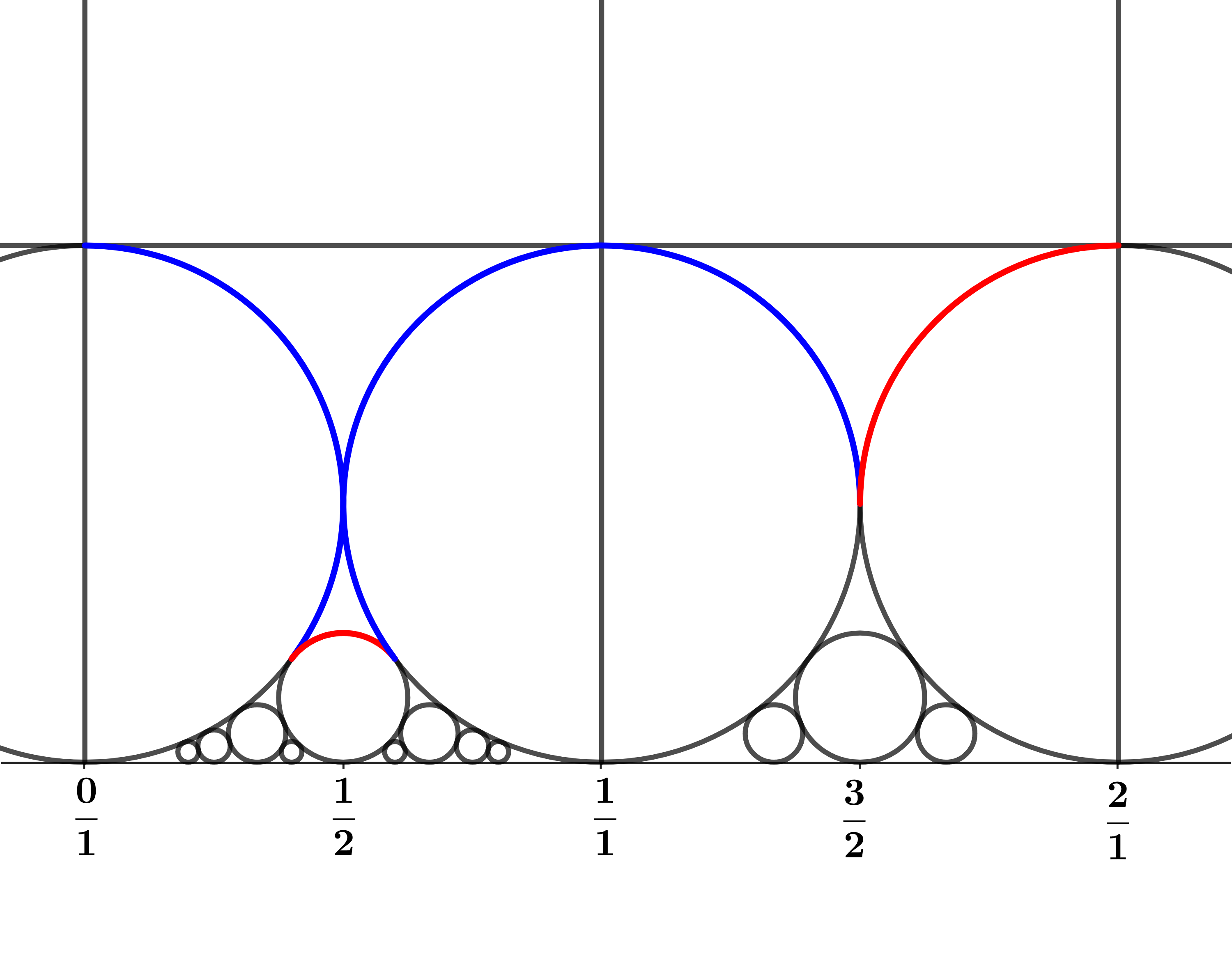}
\end{center}
\caption{Movement on the third level with [\textcolor{blue}{$h^+_2$}] (the descent) followed by \textcolor{red}{$h^-_1$}, then \textcolor{blue}{$h^+_2$} and lastly \textcolor{red}{$h^-_1$}}
%\caption{Movement on the third level with [$n^+_2$ (blue)] followed by $n^-_1$ (red), then $n^+_2$ and lastly $n^-_1$}
\label{fig:Livello2}
\end{figure}
\newpage
\begin{figure}[ht]
\begin{center}
\includegraphics[width=13cm, trim =0.1mm 2mm 0.4mm 3.5mm, clip]{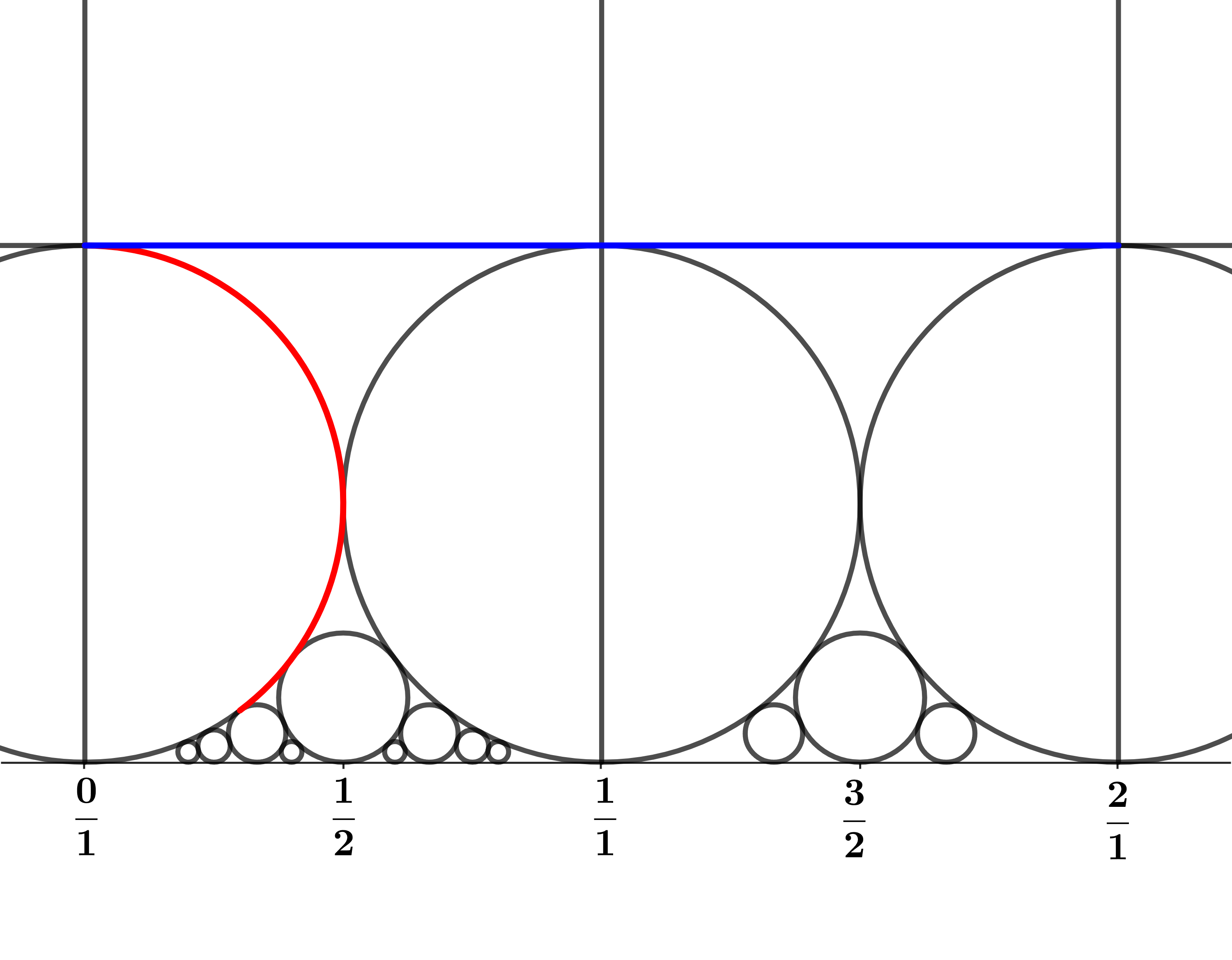}
\end{center}
\caption{Transition to the fourth level with \textcolor{blue}{$h^+_{-2}$} followed by \textcolor{red}{$h^-_3$}}
%\caption{Transition to the fourth level with $n^+_{-2}$ (blue) followed by $n^-_3$ (red)}
\label{fig:Acapo2}
\end{figure}
\begin{figure}[h!]
\begin{center}
\includegraphics[width=\textwidth, trim =2mm 3mm 1mm 4mm, clip]{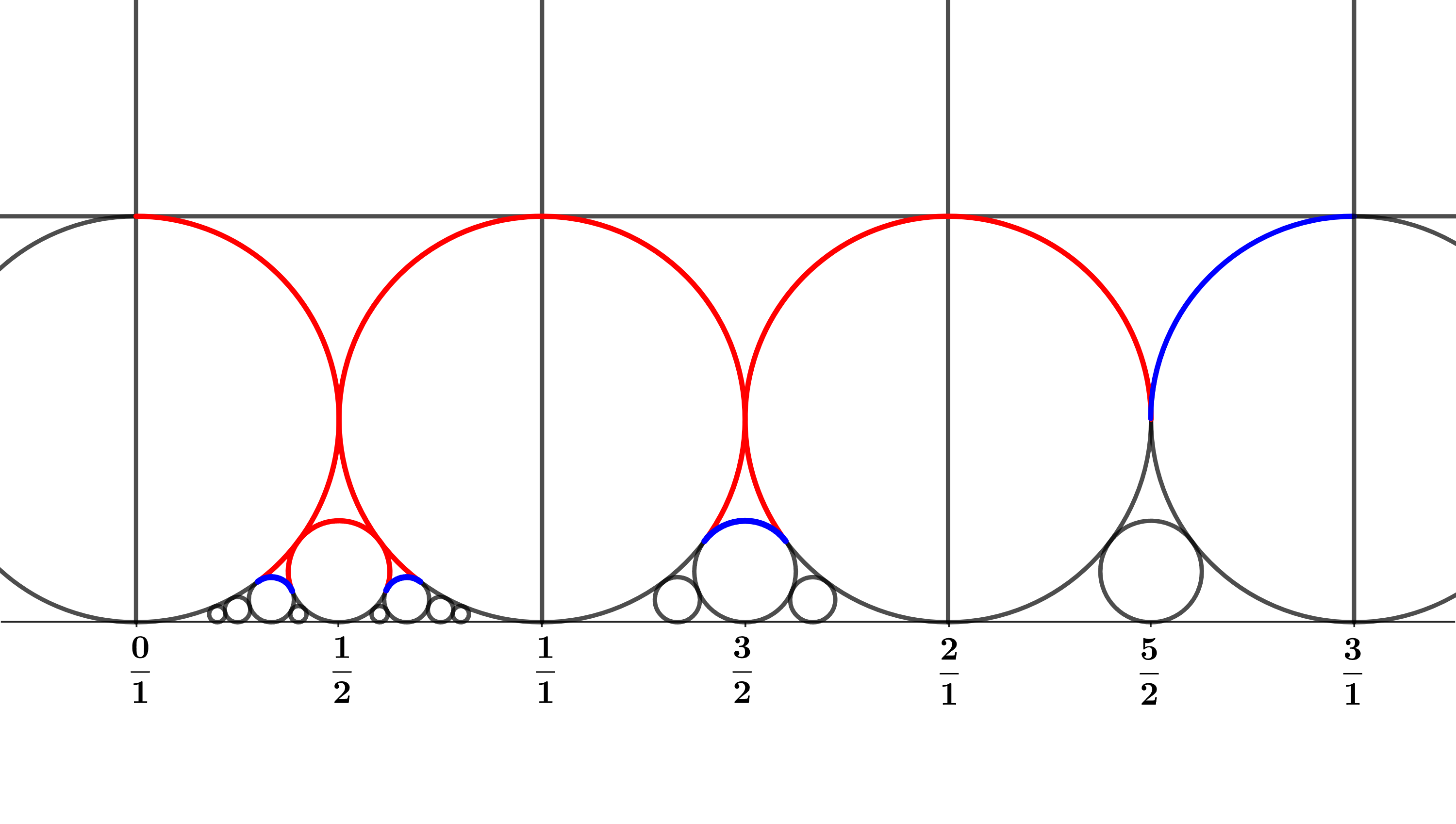}
\end{center}
\caption{Movement on the fourth level with [\textcolor{red}{$h^-_3$}] (the descent) followed by \textcolor{blue}{$h^+_{-1}$}, then \textcolor{red}{$h^-_3$}, then \textcolor{blue}{$h^+_{-1}$}, then \textcolor{red}{$h^-_5$}, then \textcolor{blue}{$h^+_{-1}$}, then \textcolor{red}{$h^-_3$}, and lastly \textcolor{blue}{$h^+_{-1}$}}
%\caption{Movement on the fourth level with [$n^-_3$ (red)] followed by $n^+_{-1}$ (blue), then $n^-_3$, then $n^+_{-1}$, then $n^-_5$, then $n^+_{-1}$, then $n^-_3$, and lastly $n^+_{-1}$}
\label{fig:Livello3}
\end{figure}

\end{document}